\documentclass{amsart}
\usepackage{amsthm, amsmath, amscd, amssymb, latexsym, stmaryrd, color}%txfonts 
\theoremstyle{plain}

\newtheorem{theorem}{Theorem}[section]
\newtheorem{lemma}[theorem]{Lemma}
\newtheorem{proposition}[theorem]{Proposition}
\newtheorem{prop-def}[theorem]{Proposition-Definition}
\newtheorem{corollary}[theorem]{Corollary}

\theoremstyle{definition}

\newtheorem{definition}[theorem]{Definition}

\newtheorem{remark}[theorem]{Remark}
\theoremstyle{remark}
\newtheorem*{ack}{Acknowledgement}
\usepackage[mathscr]{eucal}
\usepackage{graphics, graphpap}
\usepackage{array, tabularx, longtable}
\usepackage{color}
\numberwithin{equation}{section}

\def\Var{\mathrm{Var}}

\def\ord{\mathrm{ord}}

\def\id{\mathrm{id}}

\def\H{\mathscr{H}}

\def\Spec{\mathrm{Spec}}

\def\m{\mathbf m}

\def\gcd{\mathrm{gcd}}

\def\pr{\mathrm{pr}}
\def\loc{\mathrm{loc}}
\def\inte{\mathrm{int}}
\def\sr{\mathrm{sr}}
\def\ssr{\mathrm{ssr}}

\def\g{\mathbf{g}}
\def\p{\mathbf{p}}
\def\e{\mathbf{e}}

\def\MM{\overline{\mathscr{M}}}

\def\Lbb{\mathbb L}

\def\T{\mathbf T}
\def\U{\mathbf U}
\def\V{\mathbf V}
\def\f{\mathbf f}
\def\l{\mathbf l}

\def\k{\mathbf{k}}

\def\a{\mathbf{a}}
\def\b{\mathbf{b}}

\def\n{\mathbf{n}}

\begin{document} 
%-------------------
\title[Motivic Euler reflexion formulas]{\bf Euler reflexion formulas for motivic multiple zeta functions}  % Declares the document's title.

\author{L\^e Quy Thuong}
\address{Department of Mathematics, Vietnam National University \newline \indent 
334 Nguyen Trai Street, Thanh Xuan District, Hanoi, Vietnam}
\email{leqthuong@gmail.com}
\thanks{The first author's research is funded by the Vietnam National University, Hanoi (VNU) under project number QG.16.06.}

\author{Nguyen Hong Duc}
\address{Quang Binh University\newline \indent 312 Ly Thuong Kiet, Dong Hoi City, Quang Binh, Vietnam.} 
\email{nhduc82@gmail.com}
\curraddr{BCAM -- Basque Center for Applied Mathematics \newline \indent 
Mazarredo, 14, 48009 Bilbao, Basque Country -- Spain}
\email{hnguyen@bcamath.org}
\thanks{The second author was supported by Vietnam National Foundation for Science and Technology Development (NAFOSTED) under Grant Number 101.04-2014.23.}

\thanks{This research is also supported by ERCEA Consolidator Grant 615655 - NMST and by the Basque Government through the BERC 2014-2017 program and by Spanish Ministry of Economy and Competitiveness MINECO: BCAM Severo Ochoa excellence accreditation SEV-2013-0323.}


\date{-- and, in revised form, --.}


          % End of preamble and beginning of text. 
\begin{abstract}
We introduce a new notion of $\boxast$-product of two integrable series with coefficients in distinct Grothendieck rings of algebraic varieties, preserving the integrability and commuting with the limit of rational series. In the same context, we define a motivic multiple zeta function with respect to an ordered family of regular functions, which is integrable and connects closely to Denef-Loeser's motivic zeta functions. We also show that the $\boxast$-product is associative in the class of motivic multiple zeta functions.

Furthermore, a version of the Euler reflexion formula for motivic zeta functions is nicely formulated to deal with the $\boxast$-product and motivic multiple zeta functions, and it is proved for both univariate and multivariate cases by using the theory of arc spaces. As an application, taking the limit for the motivic Euler reflexion formula we recover the well known motivic Thom-Sebastiani theorem.  
\end{abstract}

\maketitle                 % Produces the title.

\section{Introduction}
We study extensions of Denef-Loeser's motivic zeta functions under motivations from a nice simple formula concerning multiple zeta values $\zeta$ and from a problem on poles of the Igusa local zeta function of a Thom-Sebastiani type function. The latter may involve the monodromy conjecture, the highest interest of ours so that the present work is just a start. The relation between real numbers $s_1, s_2\geq 2$ presented through the single and double zeta values as  
$$\zeta(s_1)\zeta(s_2)=\zeta(s_1,s_2)+\zeta(s_2,s_1)+\zeta(s_1+s_2)$$
is widely known as the Euler reflexion formula, whose further important generalizations can be found in Zagier's works, such as \cite{Zagier}. This beauty partially inspires us to consider an analogous phenomenon in the framework of motivic zeta functions, which probably provides more profound relations than the motivic Thom-Sebastiani theorem does. 

In \cite{Denef} and \cite{DV}, Denef and Denef-Veys discuss poles of the Igusa local zeta function $Z_{\Phi}(s,\chi,f)$ of a polynomial $f$ with respect to a Schwartz-Bruhat function $\Phi$ and to a character $\chi$. It is proved that there exists a function $A(s,\chi)$ depending on a character such that, for polynomials $f$ and $g$ and Schwartz-Bruhat functions $\Phi$ and $\Psi$, the poles of $A(s,\chi)Z_{\Phi\Psi}(s,\chi,f(x)+g(y))$ are of the form $s_1+s_2$, where $s_1$ and $s_2$ are poles of $A(s,\chi_1)Z_{\Phi}(s,\chi_1,f)$ and $A(s,\chi_2)Z_{\Psi}(s,\chi_2,g)$, respectively, for some $\chi_1\chi_2=\chi$. Naturally, we can ask whether a similar result still holds for motivic zeta functions, and, hopefully, a motivic Euler reflexion formula may be the first step to answer it.

The motivic zeta function of a regular function was developed in the background of Denef-Loeser's motivic integration \cite{DL1, DL2, DL3}. Afterwards, a version for a family of regular functions was also discussed in \cite{G} and \cite{GLM2}. Such a motivic zeta function for $r$ regular functions $f_i$ on a smooth algebraic variety $X$ over a field $k$ of characteristic zero is a formal series $Z_{f_1,\dots,f_r}(T_1,\dots,T_r)$ with coefficients in a certain monodromic Grothendieck ring $\mathscr M_{X_0}^{\hat\mu}$, where $X_0$ is the common zero set of the family of $f_i$. Originally, it is defined as follows
$$Z_{f_1,\dots,f_r}(T_1,\dots,T_r)=\sum\ [\mathscr X_{n_1,\dots,n_r}]\Lbb^{-d\sum n_i}\ T_1^{n_1}\cdots T_r^{n_r},$$ 
where the sum is taken over $\mathbb N_{>0}^r$ and $\mathscr X_{n_1,\dots,n_r}$ is the set of arcs $\varphi\in \mathscr L_{\sum n_i}(X)$ such that $f_i(\varphi)=t^{n_i}$ modulo $t^{n_i+1}$. When looking for a motivic analogue of the Euler reflexion formula, we recognize that $Z_{f_1,\dots,f_r}$ is still rather far to be an appropriate one, even letting the sum run over the ``optimal'' subset $\Delta$ of $\mathbb N_{>0}^r$ defined by $1\leq n_1<\cdots <n_r$. This requires a solid improvement in many aspects, including motivic zeta functions and products of them. In our approach, we replace the conditions $f_i(\varphi)=t^{n_i}$ modulo $t^{n_i+1}$ by $\ord f_i>n_i$ for every $2\leq i\leq r$, and take the sum over $\Delta$, where the resulting motivic zeta function will be denoted by $\zeta_{f_1,\dots,f_r}(T_1,\dots,T_r)$. This new notation still covers classical motivic zeta functions $Z_{f_1}(T_1)$, thus from now on we shall write $\zeta_{f_1}(T_1)$ in stead of $Z_{f_1}(T_1)$ for the coherence in literature. The integrability of $\zeta_{f_1,\dots,f_r}(T_1,\dots,T_r)$ will be proved in Corollary \ref{integrabilitymmzf}.

We introduce a new product of two integrable series (e.g., motivic zeta functions) in different rings of formal series. More precisely, if $a(\T)\in \MM_X^{\hat\mu}[[\T]]$ and $b(\U)\in \MM_Y^{\hat\mu}[[\U]]$ are integrable series in several variables, we define a reasonable element $a(\T)\boxast b(\U)$ in $\MM_{X\times Y}^{\hat\mu}[[\T,\U]]$ which is also an integrable series (Definitions \ref{2-product} and \ref{n-product}, Corollary \ref{integrabilitymmzf}). Here, for a technical reason, we work in an appropriate localization $\MM_X^{\hat\mu}$ of $\mathscr M_X^{\hat\mu}$ for any base $X$. Roughly speaking, the $\boxast$-product is an object lying in the middle of the external product and the convolution. When $\T$ and $\U$ reduce to univariates $T$ and $U$, the commuting of $\boxast$ with $\lim_{T=U\to\infty}$ will be stated in Theorem \ref{commutativity} and given a complete proof. This product allows us to describe the motivic zeta function of a Thom-Sebastiani type regular function in terms of motivic multiple zeta functions. 

The following is the statement of the most important results of the present article, the motivic Euler reflexion formulas. Let $X$ and $Y$ be smooth algebraic $k$-varieties, on which it admits regular functions $f$ and $g$ with the zero loci $X_0$ and $Y_0$, respectively. Let $f\oplus g$ be the function on $X\times Y$ defined by the sum $f(x)+g(y)$. Denote by $\iota$ the inclusion of $X_0\times Y_0$ in $X\times Y$. The motivic Euler reflexion formula in this case states that the identity
$$\zeta_f(T)\boxast \zeta_g(U)=\zeta_{f,g}(T,U)+\zeta_{g,f}(U,T)+\iota^*\zeta_{f\oplus g}(TU),$$
holds in $\MM_{X_0\times Y_0}^{\hat\mu}[[T,U]]$. This formula is given in Theorem \ref{thm21}. As an application, taking $T=U$ and using the fact that $\boxast$ and $\lim_{T\to\infty}$ commute, we can deduce from the motivic Euler reflexion formula the motivic Thom-Sebastiani theorem, which was proved previously in \cite{DL3}, \cite{Loo} and \cite{Thuong2}. 

More generally, we consider ordered families of regular functions $\f=(f_1,\dots, f_r)$ and $\g=(g_1,\dots, g_s)$ on smooth algebraic $k$-varieties $X_1,\dots, X_r$ and $Y_1,\dots, Y_s$, with common zero loci $X_0$ and $Y_0$, respectively, and formulate the general motivic Euler reflexion formula as follows
\begin{equation*}
\zeta_{\f}(\T)\boxast \zeta_{\g}(\U)
=\sum \iota^*\zeta_{p_1,\dots, p_{\eta}}(T_{\alpha_1}^{a_1}U_{\beta_1}^{b_1},\dots,T_{\alpha_{\eta}}^{a_{\eta}}U_{\beta_{\eta}}^{b_{\eta}}),
\end{equation*}
where the context of the identity is $\MM_{X_0\times Y_0}^{\hat\mu}[[\T,\U]]$, and the sum is taken over all the ordered families of regular functions $(p_1,\ldots,p_{\eta})$ satisfying 
$$p_i=a_i f_{\alpha_i}\oplus b_i g_{\beta_i},\ 1\leq i\leq \eta,$$
with $(a_i,b_i)\in \{0,1\}^2\setminus \{(0,0)\}$, $\sum(a_i+b_i)=r+s$, and $\{\alpha_i\}_{a_i=1}$ and $\{\beta_i\}_{b_i=1}$ being strictly monotonic increasing sequences, and $\iota$ is the inclusion of $X_0\times Y_0$ in $\prod_{i=1}^rX_i\times \prod_{j=1}^sY_j$ (see Theorem \ref{associativity}). An direct corollary of this formula is the associativity of the $\boxast$-product in the class of motivic multiple zeta functions (see Corollaries \ref{lastcor1} and \ref{lastcor2}).

\begin{ack}
Part of the present article was written at the Mathematisches Forschungsinstitut Oberwolfach (MFO) and the Vietnam Institute for Advanced Study in Mathematics (VIASM). Thanks are sincerely sent to the institutes for their warm hospitality during the authors' visits. 
\end{ack}

%****************************
\section{Preliminaries}
\subsection{Grothendieck rings and rings of formal series}
Let $k$ be a field of characteristic zero, $X$ an algebraic $k$-variety and $\Var_X$ the category of $X$-varieties. The {\it Grothendieck group} $K_0(\Var_X)$ of $X$-varieties is an abelian group generated by symbols $[Y\to X]$ for objects $Y\to X$ in $\Var_X$ modulo the following relations 
$$[Y\to X]=[V\to X]$$ 
if $Y\to X$ and $V\to X$ are isomorphic in $\Var_X$, and 
$$[Y\to X]=[V\to X]+[Y\setminus V\to X]$$ 
if $V$ is Zariski closed in $Y$. Furthermore, $K_0(\Var_X)$ has structure of a ring with unit with product induced by fiber product of $X$-varieties and the unit being the class of the identity morphism $X\to X$. Let $\mathscr M_X$ be the localization of $K_0(\Var_X)$ with respect to the multiplicative system of $\Lbb^i$ with $i\in\mathbb N$, where $\Lbb:=[\mathbb A_X^1]=[\mathbb A_k^1\times X\to X]$. In this situation and from now on, whenever writing $X\times X'$ for $k$-schemes $X$ and $X'$ we means the fiber product $X\times_k X'$.

Let $\mu_n=\mu_n(k)$ be the group scheme of $n$th roots of unity in $k$, $\Spec(k[t]/(t^n-1))$. The family of all $\mu_n$, $n\in\mathbb N_{>0}$, forms a projective system with respect to morphisms $\mu_{nm}\to\mu_{n}$ given by $\xi\mapsto \xi^m$, we denote its projective limit by $\hat{\mu}$. By definition, a {\it good} $\mu_n$-action on an $X$-variety $Y$ is a group action $\mu_n\times Y \to Y$, which is a morphism of $X$-varieties, such that each orbit is contained in an affine $k$-subvariety of $Y$; a {\it good} $\hat{\mu}$-action on $Y$ is an action of $\hat{\mu}$ on $Y$ factoring through a good $\mu_n$-action.

The {\it monodromic Grothendieck group} $K_0^{\hat{\mu}}(\Var_X)$ of $X$-varieties endowed with good $\hat\mu$-action is an abelian group generated by the $\hat\mu$-equivariant isomorphism classes $[Y\to X,\sigma]$, $\sigma$ being a good $\hat{\mu}$-action on $X$-variety $Y$, modulo the following conditions 
$$[Y\to X,\sigma]=[V\to X,\sigma|_V]+[Y\setminus V\to X,\sigma|_{Y\setminus V}]$$ 
if $V$ is Zariski closed in $Y$ and 
$$[Y\times\mathbb A_k^n\to X,\sigma]=[Y\times\mathbb A_k^n\to X,\sigma']$$ 
if $\sigma$, $\sigma'$ lift the same $\hat{\mu}$-action on $Y\to X$ to an affine action on $Y\times\mathbb A_k^n\to X$. When no confusion may happen, we write $[Y,\sigma]$ for $[Y\to X,\sigma]$ for simplicity. Thanks to fiber product of $X$-varieties, $K_0^{\hat{\mu}}(\Var_X)$ has the natural structure of a ring. Define 
$$\mathscr M_X^{\hat{\mu}}:=K_0^{\hat{\mu}}(\Var_X)[\Lbb^{-1}],$$ 
the $\hat{\mu}$-equivariant version of the ring $\mathscr M_X$. We also consider the ring $\mathscr M_X^{\hat{\mu}^r}$ when working with good $\hat\mu^r$-actions. Let $\MM_X^{\hat{\mu}}$ be the localization of $\mathscr M_X^{\hat{\mu}}$ with respect to the multiplicative family generated by the elements $1-\Lbb^n$ with $n\in\mathbb N_{>0}$. There is a natural morphism $\loc: \mathscr M_X^{\hat{\mu}}\to \MM_X^{\hat{\mu}}$, which has not been proved or disproved to be injective; however, for simplicity of notation, if necessary, we shall identify $a$ with $\loc(a)$, that is, consider $a\in \mathscr M_X^{\hat{\mu}}$ as an element of $\MM_X^{\hat{\mu}}$. 
 
For a morphism of $ k$-varieties $f: X\to X'$, one defines group morphisms $f_!: \mathscr M_X^{\hat{\mu}}\to \mathscr M_{X'}^{\hat{\mu}}$ and $f_!: \MM_X^{\hat{\mu}}\to \MM_{X'}^{\hat{\mu}}$ by composition, also defines ring morphisms $f^*: \mathscr M_{X'}^{\hat{\mu}}\to \mathscr M_X^{\hat{\mu}}$ and $f^*: \MM_{X'}^{\hat{\mu}}\to \MM_X^{\hat{\mu}}$ by fiber product. If $X'=\Spec k$, $f_!$ is usually denoted by $\int_X$.

Let $\mathscr M$ be a $\mathbb Z[\Lbb,\Lbb^{-1}]$-module, and let $\T=(T_1,\dots,T_r)$ be a multivariate. We shall consider $\mathscr M[[\T]]$ and the following sub-$\mathbb Z[\Lbb,\Lbb^{-1}]$-modules
\begin{align*}
\mathscr M[[\T]]_{\sr}\ &:=\mathscr M[\T]\left[(1-\Lbb^m\T^{\n})^{-1}\right]_{(m,\n)\in \mathbb Z\times (\mathbb N^r\setminus\{(0,\dots,0)\})},\\
\mathscr M[[\T]]_{\ssr}&:=\mathscr M[\T]\left[(1-\Lbb^m\T^{\n})^{-1}\right]_{(m,\n)\in \mathbb Z_{\leq 0}\times (\mathbb N^r\setminus\{(0,\dots,0)\})},\ \text{and}\\
\mathscr M[[\T]]_{\inte}&:=\mathscr M[\T]\left[(1-\Lbb^m\T^{\n})^{-1}\right]_{(m,\n)\in \mathbb Z_{<0}\times (\mathbb N^r\setminus\{(0,\dots,0)\})}.
\end{align*}
The identity 
$$\frac 1 {1-\Lbb^m\T^{\n}}=\sum_{l\geq 0}(\Lbb^m\T^{\n})^l$$
induces canonical embeddings of the previous modules in $\mathscr M[[\T]]$. Elements of $\mathscr M[[\T]]_{\sr}$ are called {\it rational series}, elements of $\mathscr M[[\T]]_{\ssr}$ are called {\it strongly rational series}, and elements of $\mathscr M[[\T]]_{\inte}$ are called {\it integrable series}, over $\mathscr M$. It is immediate that an integrable series is also a strongly rational series and a strongly rational series is also a rational series. The terminology ``integrable'' is inspired from the discussions of Cluckers and Loeser on integrable constructible functions in Section 4, especially Theorem 4.5.4, of their article \cite{CL}. 

In particular, if we fix a $k$-variety $X$ and let $\mathscr M$ be one of two rings $\mathscr M_X^{\hat{\mu}}$ and $\MM_X^{\hat{\mu}}$, then the previous rings can be obviously viewed as $\mathscr M$-modules. If this is the case, and if $\T$ reduces to a univariate $T$, we get that every integrable series is also {\it of finite mass} in the sense of Looijenga \cite{Loo}. Moreover, as shown in \cite{DL1}, there exists a unique $\mathscr M_X^{\hat{\mu}}$-linear morphism 
$$\lim_{T\to\infty}: \mathscr M[[T]]_{\sr}\to \mathscr M$$ 
such that $\lim_{T\to\infty}\frac{\Lbb^mT^n}{1-\Lbb^mT^n}=-1$ for any $(m,n)\in\mathbb Z\times \mathbb N_{>0}$.

%---------------------------
\subsection{Arc spaces and motivic zeta functions}\label{subsec2.2}
Let $X$ be an algebraic $k$-variety. For any $n\in\mathbb N_{>0}$, let $\mathscr L_n(X)$ be the space of $n$-jet schemes of $X$, which is a $k$-scheme representing the functor sending a $k$-algebra $A$ to the set of morphisms of $k$-schemes $\Spec(A[t]/(t^{n+1}))\to X$. For any pair $n\leq m$, the truncation defines a morphism of $k$-schemes 
$$\pi_n^m:\mathscr L_m(X)\to \mathscr L_n(X)$$ 
and this is an affine morphism. If $X$ is smooth of dimension $d$, the morphism $\pi_n^m$ is a locally trivial fibration with fiber $\mathbb{A}_k^{(m-n)d}$. The $n$-jet schemes and the morphisms $\pi_n^m$ form in a natural way a projective system of $k$-schemes, we denote its limit by $\mathscr L(X)$ and call this space {\it the arc space of $X$}. For any field extension $k\subset K$, the $K$-points of $\mathscr{L}(X)$ correspond one-to-one to the $K[[t]]$-points of $X$. %We denote by $\pi_m$ the natural morphism $\mathcal{L}(\mathscr X)\to\mathcal{L}_m(\mathscr X)$. 

Furthermore, the schemes $\mathscr L_n(X)$ and $\mathscr L(X)$ are endowed with a natural action of $\mu_n$ given by $\xi\varphi(t):=\varphi(\xi t)$. The profinite group scheme $\hat\mu$ acts on these schemes via $\mu_n$'s.

%---------------------------

Assume in the rest of this section that $X$ is a smooth algebraic $k$-variety of pure dimension $d$. Let $f:X\to\mathbb A_k^1$ be a regular function with the zero locus $X_0$. For $n\in\mathbb N_{>0}$, let $\mathscr{X}_n(f)$ be the set of arcs $\varphi\in\mathscr{L}_n(X)$ such that $f(\varphi)= t^n\mod t^{n+1}$. Since the image of $\mathscr{X}_n(f)$ under the canonical morphism $\mathscr L_n(X)\to X$ is contained in $X_0$, it is also an $X_0$-variety. Furthermore, $\mathscr{X}_n(f)$ is stable for the action of $\mu_n$ on $\mathscr L_n(X)$, thus it defines a motivic class $[\mathscr{X}_n(f)]:=[\mathscr{X}_n(f)\to X_0]$ in $\mathscr M_{X_0}^{\hat{\mu}}$. The {\it motivic zeta function of $f$} is defined as follows
\begin{align*}%\label{zeta} 
Z_f(T):=\sum_{n\geq 1}[\mathscr{X}_n(f)]\Lbb^{-nd}T^n,
\end{align*}
which lives in $\mathscr M_{X_0}^{\hat{\mu}}[[T]]$. If $x$ is a closed point in $X_0$, we define the {\it local motivic zeta function} $Z_{f,x}(T)$ to be $x^*Z_f(T)$, where $x^*$ stands for the pullback of the inclusion of $x$ in $X_0$. Clearly, the series $Z_{f,x}(T)$ is an element of $\mathscr M_k^{\hat{\mu}}[[T]]$.

\begin{theorem}[Denef-Loeser]\label{DL98}
The motivic zeta function $Z_f(T)$ is an integrable series. 
\end{theorem}

The proof of Theorem \ref{DL98} by Denef and Loeser in \cite{DL1} uses in a crucial way invariants of a log-resolution of $X_0$. Let us now recall briefly their work with such a resolution $h: Y\to (X,X_0)$. The exceptional divisors and irreducible components of the strict transform for $h$ will be denoted by $E_i$, where $i$ is in a finite set $A$. For $\emptyset\not=I\subset A$, one puts 
$$E_I=\bigcap_{i\in I}E_i\ \ \text{and}\ \ E_I^{\circ}=E_I\setminus\bigcup_{j\not\in I}E_j.$$ 
Consider an affine covering $\{U\}$ of $Y$ such that on each piece $U\cap E_I^{\circ}\not=\emptyset$ the pullback of $f$ has the form $u\prod_{i\in I}y_i^{N_i}$ with $u$ a unit and $y_i$ a local coordinate defining $E_i$. Let $m_I$ denote $\gcd(N_i)_{i\in I}$. Denef and Loeser study the unramified Galois covering $\pi_I:\widetilde{E}_I^{\circ}\to E_I^{\circ}$ with Galois group $\mu_{m_I}$ defined locally with respect to $\{U\}$ as follows 
$$\left\{(z,y)\in \mathbb{A}_k^1\times(U\cap E_I^{\circ}) \mid z^{m_I}=u(y)^{-1}\right\}.$$ 
The local pieces are glued over $\{U\}$ as in the proof of \cite[Lemma 3.2.2]{DL1} to give $\widetilde{E}_I^{\circ}$ and $\pi_I$ as mentioned, and the definition of the covering $\pi_I$ is independent of the choice of the affine covering $\{U\}$. Furthermore, $\widetilde{E}_I^{\circ}$ is endowed with a $\mu_{m_I}$-action by multiplication of the $z$-coordinate with elements of $\mu_{m_I}$, defining a class $[\widetilde{E}_I^{\circ}]=[\widetilde{E}_I^{\circ}\to E_I^{\circ}\to X_0]$ in $\mathscr M_{X_0}^{\hat\mu}$ (cf. \cite{DL5}). For each $i\in A$, we denote by $\nu_i-1$ the multiplicity of $E_i$ in the divisor of $h^*\omega_X$, where $\omega_X$ is a local generator of the sheaf of differential forms on $X$ of maximal degree. Then Denef-Loeser's formula of motivic zeta function in terms of $h$ is the following 
\begin{align}\label{DL-form}
Z_f(T)=\sum_{\emptyset\not=I\subset A}(\mathbb{L}-1)^{|I|-1}[\widetilde{E}_I^{\circ}]\prod_{i\in I}\frac{\Lbb^{-\nu_i}T^{N_i}}{1-\Lbb^{-\nu_i}T^{N_i}},
\end{align}
which holds in $\mathscr M_{X_0}^{\widehat{\mu}}[[T]]$. This proves that $Z_f(T)$ is an integrable series.

The quantity 
$$\mathscr S_f:=-\lim_{T\to\infty}Z_f(T)=\sum_{\emptyset\not=I\subset A}(1-\mathbb{L})^{|I|-1}[\widetilde{E}_I^{\circ}]$$ 
in $\mathscr M_{X_0}^{\hat{\mu}}$ is called {\it the motivic nearby cycles of $f$}. Also, the element $\mathscr S_{f,x}:=x^*\mathscr S_f$ of $\mathscr M_k^{\hat{\mu}}$ is called {\it the motivic Milnor fiber of $f$ at $x$}. Recently, $\mathscr S_f$ and $\mathscr S_{f,x}$ have been getting more important in singularity theory because of their relations with various classical invariants, such as Euler characteristic, Hodge spectrum, monodromy zeta functions (cf. \cite{DL4, DL5}, \cite{G}, \cite{GLM1}).

More generally, we are going to consider a modification of the motivic zeta function in several variables concerning a family of functions mentioned in Guibert \cite{G}. The version with a rational polyhedral convex cone in $\mathbb N_{>0}^r$ was studied by Guibert-Loeser-Merle \cite{GLM2} for one variable with respect to an appropriate linear form on the cone. Let $\mathbf f$ be an ordered family of $r$ regular functions $f_i: X\rightarrow \mathbb A_k^1$. For simplicity of notation, we also write $X_0$ for $X_0(\mathbf f)$, the common zeros of the family $\mathbf f$. For any $\n\in \mathbb N_{>0}^r$, let $|\n|=\sum_{i=1}^rn_i$, and we define
$$\mathscr{X}_{\n}(\mathbf f):=\left\{\varphi\in\mathscr{L}_{|\n|}(X)\mid f_i(\varphi)= t^{n_i}\mod t^{n_i+1}, 1\leq i\leq r\right\}.$$
In the particular case where $X=X_1\times\cdots\times X_r$ with $X_i$ smooth algebraic $k$-varieties and, for every $1\leq i\leq r$, $f_i$ is a regular function on $X_i$, we define   
$$\mathscr{D}_{\n}(\mathbf f):=\left\{\varphi=(\varphi_1,\dots,\varphi_r)\in\mathscr{L}_{|\n|}(X)\ \begin{array}{|l} f_1(\varphi_1)= t^{n_1}\mod t^{n_1+1}\\ \ord f_i(\varphi_i)>n_i, 2\leq i\leq r\end{array}\right\}.$$
It is clear that, for every $\n\in \mathbb N_{>0}^r$, $\mathscr{X}_{\n}(\mathbf f)$ (resp. $\mathscr{D}_{\n}(\mathbf f)$) is stable under the good $\mu_{\gcd(\n)}$-action (resp. $\mu_{n_1}$-action) on the space $\mathscr{L}_{|\n|}(X)$ given by $\xi\varphi(t):=\varphi(\xi t)$ (resp. $\xi\varphi(t):=(\varphi_1(\xi t),\varphi_2(t),\dots,\varphi_r(t)$), and that $\mathscr{X}_{\n}(\mathbf f)$ (resp. $\mathscr{D}_{\n}(\mathbf f)$) admits a morphism to $X_0$. This fact thus gives rise to an element $[\mathscr{X}_{\n}(\mathbf f)]:=[\mathscr{X}_{\n}(\mathbf f)\to X_0]$ (resp. $[\mathscr{D}_{\n}(\mathbf f)]:=[\mathscr{D}_{\n}(\mathbf f)\to X_0]$) in $\mathscr{M}_{X_0}^{\hat\mu}$. 

Let $C$ be a rational polyhedral convex cone in $\mathbb N_{>0}^r$, let $\Delta$ be the special one among $C$'s which consists of $\n=(n_1,\dots,n_r)$ such that $1\leq n_1<\cdots< n_r$. Let $\T$ denote the $r$-tuple $(T_1,\dots,T_r)$ of variables, and let $\T^{\n}$ stand for $T_1^{n_1}\cdots T_r^{n_r}$. 

\begin{definition}\label{MMZF}
The {\it motivic zeta function $Z_{\f}^C(\T)$} of a family $\f$ of regular functions on $X$ is the following series in $\mathscr M_{X_0}^{\hat\mu}[[\T]]$:  
$$Z_{\f}^C(\T):=\sum_{n\in C}[\mathscr{X}_{\n}(\mathbf f)]\Lbb^{-|\n|d}\T^{\n}.$$ 
If $\f=(f_1,\dots,f_r)$ is an ordered family $\f$ of regular functions $f_i:X_i\to \mathbb A_k^1$, $1\leq i\leq r$, the {\it motivic multiple zeta function $\zeta_{\f}(\T)$} of $\f$ is the series
$$\zeta_{\f}(\T):=\sum_{n\in {\Delta}}[\mathscr{D}_{\n}(\mathbf f)]\Lbb^{-|\n|d}\T^{\n}$$ 
in $\mathscr M_{X_0}^{\hat\mu}[[\T]]$. For a closed point $x\in X_0$, we define the {\it local motivic} and the {\it local motivic multiple zeta functions} as $Z_{\f,x}^C(\T):=x^*Z_{\mathbf f}^C(\T)$ and $\zeta_{\mathbf f,x}(\T):=x^*\zeta_{\mathbf f}(\T)$, elements of $\mathscr M_k^{\hat\mu}[[\T]]$.
\end{definition}

We refer to \cite[Section 2.9]{GLM2} and \cite[Lemma 3.4]{DL2} to see that $Z_{\f}^C(\T)$ is a rational series. Indeed, we can obtain the motivic zeta function $Z_{\f}^{C,\ell}(T)$ in \cite{GLM2}, which depends on a linear form $\ell$ positive on the closure $\overline{C}$ of $C$ in $\mathbb R_{\geq 0}^r$ except at the origin, in terms of replacing $\T^{\n}$ in $Z_{\f}^C(\T)$ by $T^{\ell(\n)}$. There, Guibert, Loeser and Merle deduce the rationality of $Z_{\f}^{C,\ell}(T)$ thanks to \cite[Lemma 3.4]{DL2}, and, fortunately, their arguments are definitely applicable to the rationality of $Z_{\f}^C(\T)$. To be more precise, let us consider a log-resolution $h: Y\to (X,X_0(F))$, with $F=f_1\cdots f_r$. Assume that 
$$h^{-1}(X_0(F))=\sum_{i\in A}N_i(F)E_i\quad \text{and} \quad K_Y=h^*K_X+\sum_{i\in A}(\nu_i-1)E_i,$$ 
where $E_i$'s are irreducible components of $h^{-1}(X)$. As previous, we shall work with
$[\widetilde{E}_I^{\circ}]$ in $\mathscr M_{X_0(F)}^{\hat\mu}$ for any nonempty $I\subset A$. Denote by $N_i(f_j)$ the multiplicity of $E_i$ in the divisor of $f_j\circ h$, and by $N_i$ the vector $(N_i(f_1),\dots,N_i(f_r))\in \mathbb N^r$. Denote by $\mathscr A$ the set of all nonempty subsets $I$ of $A$ such that $h(E_I^{\circ})\subset X_0=X_0(\f)$. For any $I\in \mathscr A$, we consider the linear morphisms $N_I: \mathbb R^I\to \mathbb R^r$ and $\nu_I: \mathbb R^I\to \mathbb R$ defined as follows: for any $\alpha=(\alpha_i)_{i\in I}\in \mathbb R^I$, $N_I(\alpha):=\sum_{i\in I}\alpha_iN_i$ and $\nu_I(\alpha):=\sum_{i\in I}\alpha_i\nu_i$. Using the same method as doing with (\ref{DL-form}) we obtain a formula for $Z_{\f}^C(\T)$, which lives in $\mathscr M_{X_0}^{\hat{\mu}}[[\T]]$, as follows 
\begin{align}\label{DL-form-version2}
Z_{\f}^C(\T)=\sum_{I\in \mathscr A}(\mathbb{L}-1)^{|I|-1}[\widetilde{E}_I^{\circ}]\sum_{\k\in N_I^{-1}(C)}\Lbb^{-\nu_I(\k)}\T^{N_I(\k)}.
\end{align}
Thus $Z_{\f}^C(\T)$ is integrable, since $\sum_{\k\in N_I^{-1}(C)}\Lbb^{-\nu_I(\k)}\T^{N_I(\k)}$ is integrable, for any $I\in \mathscr A$, i.e., $Z_{\f}^C(\T)\in \mathscr M_{X_0}^{\hat{\mu}}[[\T]]_{\inte}$.

Furthermore, with $\ell$ being a linear form on $\mathbb R^r$ positive on $\overline{C}\setminus \{0\}$, where $\overline{C}$ is the closure of $C$ in $\mathbb R_{\geq 0}^r$, it follows from (\ref{DL-form-version2}) that  
\begin{align*}%\label{DL-form-version3}
Z_{\f}^{C,\ell}(T)=\sum_{I\in \mathscr A}(\mathbb{L}-1)^{|I|-1}[\widetilde{E}_I^{\circ}]\sum_{\k\in N_I^{-1}(C)}\Lbb^{-\nu_I(\k)}T^{\ell(N_I(\k))},
\end{align*}
which means that $Z_{\f}^{C,\ell}(T)$ is integrable, thus rational, and we can take its limit
\begin{align*}%\label{DL-form-version3}
\lim_{T\rightarrow\infty}Z_{\mathbf f}^{C,\ell}(T)=\sum_{I\in \mathscr A}\chi(N_I^{-1}(C))(\mathbb{L}-1)^{|I|-1}[\widetilde{E}_I^{\circ}].
\end{align*}
Observe that the element $\lim_{T\rightarrow\infty}Z_{\mathbf f}^{C,\ell}(T)$ of $\mathscr{M}_{X_0}^{\hat\mu}$ is independent of the choice of such an $\ell$, hence one usually writes $\mathscr{S}_{\mathbf f}^C$ for it. For a closed point $x\in X_0$, we define $\mathscr{S}_{\mathbf f,x}:=x^*\mathscr{S}_{\mathbf f}$, which evidently equals the limit $\lim_{T\to \infty}x^*Z_{\f}^{C,\ell}(T)$.
 
Similarly, we have

\begin{proposition}\label{MMZ}
As an element of $\MM_{X_0}^{\hat\mu}[[\T]]$ the motivic multiple zeta function $\zeta_{\f}(\T)$ is integrable, i.e, $\zeta_{\f}(\T)\in \MM_{X_0}^{\hat\mu}[[\T]]_{\inte}$.
\end{proposition} 

To prove this proposition we may compute directly the series $\zeta_{\f}(\T)$ in terms of a resolution of singularity as done for $Z_{\f}^C(\T)$ in (\ref{DL-form-version2}), with a slight modification. More precisely, using the previous notation $\widetilde{E}_I^{\circ}$ but changing the definition so that $m_I=\gcd(N_i)_{i\in I}$ is replaced by $m'_I=\gcd(N_i)_{i\in I_1}$, where $I_1$ is the subset of $i\in I$ coming from resolution of $\{f_1=0\}$, and by convenience, $m'_I=1$ if $I_1=\emptyset$, we get
\begin{align*}
\zeta_{\f}(\T)=\sum_{I\in \mathscr A}(\mathbb{L}-1)^{|I|-1}[\widetilde{E}_I^{\circ}]\sum_{\beta\in N_I^{-1}(\{0\}\times \mathbb N_{>0}^{r-1})}\sum_{\begin{smallmatrix}\k'\in N_I^{-1}(\Delta)\\ k'_i>-\beta_i,i\in I\end{smallmatrix}}\Lbb^{-\nu_I(\k'+\beta)}\T^{N_I(\k')}.
\end{align*}
Remark that, in this setting, although $\k'$ and $\beta$ are rational vectors, $\k'+\beta$ are positively integral vectors. A generalization of Lemma 8.5.2 of \cite{HL} (see also \cite[Section 2.9]{GLM1}) shows that 
\begin{align}\label{IntSeries}  
S(I;\T):= \sum_{\beta\in N_I^{-1}(\{0\}\times \mathbb N_{>0}^{r-1})}\sum_{\begin{smallmatrix}\k'\in N_I^{-1}(\Delta)\\ k'_i>-\beta_i,i\in I\end{smallmatrix}}\Lbb^{-\nu_I(\k'+\beta)}\T^{N_I(\k')},
\end{align}
viewed as a series in $\MM_{X_0}^{\hat\mu}[[\T]]$, is integrable. This proves $\zeta_{\f}(\T)\in \MM_{X_0}^{\hat\mu}[[\T]]_{\inte}$. We also notice that we shall provide another proof for Proposition \ref{MMZ} in Corollary \ref{integrabilitymmzf}.

Furthermore, in terms of (\ref{IntSeries}), we get the following formula
$$
\lim_{T\to \infty}S(I;T,\dots,T)=\chi(N_I^{-1}(\{0\}\times \mathbb N_{>0}^{r-1}))\chi(N_I^{-1}(\Delta)),
$$
which allows to compute the limit $\lim_{T\to \infty}\zeta_{\f}(T,\dots,T)$ of the motivic multiple zeta function of $\f$ (see Definition \ref{limzeta}).
%\end{proof}

\begin{definition}\label{limzeta}
The {\it motivic multiple nearby cycles of the family $\f$} in Proposition \ref{MMZ}, denoted by $\mathscr S_{\f}$, is defined to be the element 
$$-\lim_{T\rightarrow\infty}\zeta_{\f}(T,\dots,T)=-\sum_{I\in \mathscr A}\chi(N_I^{-1}(\{0\}\times \mathbb N_{>0}^{r-1}))\chi(N_I^{-1}(\Delta))(\mathbb{L}-1)^{|I|-1}[\widetilde{E}_I^{\circ}]$$ 
of the ring $\MM_{X_0}^{\hat\mu}$. For a closed point $x\in X_0$, we set 
$$\mathscr{S}_{\f,x}:=(\{x\}\hookrightarrow X_0)^*\mathscr{S}_{\f}$$ 
and call it the {\it local motivic multiple nearby cycles of $\f$ at $x$}.
\end{definition}

%*****************************

\section{Hadamard products and $\boxast$-product}
\subsection{Convolution and Hadamard products}\label{convolution}
The standard concept of convolution product on the monodromic Grothendieck rings of algebraic varieties was given earlier in \cite{DL3}, \cite{Loo} and \cite{GLM1}. To recall it explicitly, let us consider the Fermat varieties $F_0^n$ and $F_1^n$ in $\mathbb G_{m,k}^2$ defined by the equations $u^n+v^n=0$ and $u^n+v^n=1$, respectively. Note that the varieties $F_0^n$ and $F_1^n$ admit the obvious action of $\mu_n\times\mu_n$. 

Let $X$, $Y$ and $Z$ be algebraic $k$-varieties endowed with good $\mu_n$-action. Assume that there are $\mu_n$-equivariant morphisms $Y\to X$ and $Z\to X$. Define operations in $\mathscr M_X^{\mu_n}$ as follows
\begin{gather*}
[Y\to X]\ast_0[Z\to X]:=[F_0^n\times^{\mu_n\times\mu_n}(Y\times_X Z)],\\
[Y\to X]\ast_1[Z\to X]:=[F_1^n\times^{\mu_n\times\mu_n}(Y\times_X Z)],\\
[Y\to X]\ast [Z\to X]:= [Y\to X]\ast_0[Z\to X]-[Y\to X]\ast_1[Z\to X],
\end{gather*}
where, for $i\in\{0,1\}$, $F_i^n\times^{\mu_n\times\mu_n}(Y\times_X Z)$ is the quotient of $F_i^n\times(Y \times_X Z)$ with respect to the equivalence relation by which any two elements $(\xi u,\eta v,x,y)$ and $(u,v,\xi x,\eta y)$ are equivalent, for $\xi, \eta \in\mu_n$. The group scheme $\mu_n$ acts diagonally on $F_i^n\times^{\mu_n\times\mu_n}(Y\times_X Z)$. Then passing through the projective limit with respect to $n\in\mathbb N_{>0}$ we get the {\it (standard) convolution product $\ast$} on $\mathscr M_X^{\hat{\mu}}$. We can also extend the $\ast$-product to $\MM_X^{\hat{\mu}}$ in a natural way. By \cite[Proposition 5.2]{GLM1}, the convolution product $\ast$ is commutative and associative.

Let $X$, $Y$, $Z$ and $W$ be algebraic $k$-varieties which are endowed with good $\hat\mu$-action and admit $\hat\mu$-equivariant morphisms $Z\to X$ and $W\to Y$ (we may choose the trivial action of $\hat\mu$ on the bases $X$ and $Y$). The cartesian product induces a morphism of rings $\mathscr M_X^{\hat{\mu}}\times \mathscr M_Y^{\hat{\mu}}\to \mathscr M_{X\times Y}^{\hat{\mu}^2}$, by which the diagonal action induces naturally a canonical morphism $\mathscr M_{X\times Y}^{\hat{\mu}^2}\to \mathscr M_{X\times Y}^{\hat{\mu}}$. Then the composition of these morphisms yields an {\it external product}
\begin{align}\label{form3.1}
\mathscr M_X^{\hat{\mu}}\times \mathscr M_Y^{\hat{\mu}}\to \mathscr M_{X\times Y}^{\hat{\mu}},
\end{align}
where, by abuse of notation, we also denote it by $\times$. As previous, we let $\T$ be an $r$-tuple of variables. The {\it (external) Hadamard $\times_{\H}$-product of two series $a(\T)=\sum_{\n\in\mathbb N^r}a_{\n}\T^{\n}$ in $\mathscr M_X^{\hat\mu}[[\T]]$ and $b(\T)=\sum_{\n\in\mathbb N^r}b_{\n} \T^{\n}$ in $\mathscr M_Y^{\hat\mu}[[\T]]$} is the series
\begin{align}\label{form3.2}
a(\T)\times_{\H} b(\T):=\sum_{\n \in\mathbb N^r}a_{\n}\times b_{\n}\T^{\n}
\end{align}
in $\mathscr M_{X\times Y}^{\hat\mu}[[\T]]$. This product is commutative, and it is also associative in the following sense, where the verification is straightforward. If $a(\T)$ is in $\mathscr M_X^{\hat\mu}[[\T]]$, $b(\T)$ is in $\mathscr M_Y^{\hat\mu}[[\T]]$ and $c(\T)$ is in $\mathscr M_Z^{\hat\mu}[[\T]]$, then the identity
\begin{align}\label{form3.3}
\left(a(\T)\times_{\H} b(\T)\right)\times_{\H} c(\T)=a(\T)\times_{\H} \left(b(\T)\times_{\H} c(\T)\right)
\end{align}
holds in $\mathscr M_{X\times Y\times Z}^{\hat\mu}[[\T]]$. It is stated similarly as in Lemma 7.6 of \cite{Loo} that, in the univariate case (i.e., $r=1$), the $\times$-product is ``anti-compatible'' with the Hadamard $\times_{\H}$-product via the morphism $\lim_{T\to\infty}$. Namely, if $a(T)$ is in $\mathscr M_X^{\hat\mu}[[T]]_{\sr}$ and $b(T)$ is in $\mathscr M_Y^{\hat\mu}[[T]]_{\sr}$, then $a(T)\times_{\H} b(T)$ is in $\mathscr M_{X\times Y}^{\hat\mu}[[T]]_{\sr}$ and the identity
\begin{align}\label{form3.4}
\lim_{T\to\infty}\left(a(T)\times_{\H} b(T)\right)=-\Big(\lim_{T\to\infty}a(T)\Big) \times \Big(\lim_{T\to\infty}b(T)\Big)
\end{align}
holds in $\mathscr M_{X\times Y}^{\hat\mu}$. An analogous assertion for an arbitrary $r$ is also true when we replace the morphism $\lim_{T\to\infty}$ by the morphism $\lim_{T_1=\cdots=T_r\to\infty}$, the composition of $\lim_{T\to\infty}$ and the assignment $T=T_1=\cdots=T_r$.% The latter can be defined in an obvious way based on $\lim_{T\to\infty}$.

The previous external product also deduces naturally the following external product, which we again denote by $\times$,
\begin{align*}%\label{form3.5}
\MM_X^{\hat{\mu}}\times \MM_Y^{\hat{\mu}}\to \MM_{X\times Y}^{\hat{\mu}}.
\end{align*}
This product has the same properties as the previous ones that we have mentioned.

Let us now introduce a generalized (external) convolution product of the previous standard one. Using the external product, the {\it generalized (external) convolution product}
$$\ast: \mathscr M_X^{\hat{\mu}}\times \mathscr M_Y^{\hat{\mu}}\to \mathscr M_{X\times Y}^{\hat{\mu}}$$
(again by abuse of notation) is defined as follows
\begin{align*}
[Z\to X]\ast_0[W\to Y]:= &\left([Z\to X]\times[Y\to Y]\right)\ast_0\left([X\to X]\times[W\to Y]\right),\\
[Z\to X]\ast_1[W\to Y]:= &\left([Z\to X]\times[Y\to Y]\right)\ast_1\left([X\to X]\times[W\to Y]\right),\\
[Z\to X]\ast [W\to Y]:= &[Z\to X]\ast_0[W\to Y]-[Z\to X]\ast_1[W\to Y].
\end{align*}
The {\it Hadamard $\ast_{\H}$-product of two formal series $a(\T)=\sum_{\n\in\mathbb N^r}a_{\n}\T^{\n}\in\mathscr M_X^{\hat\mu}[[\T]]$ and $b(\T)=\sum_{\n\in\mathbb N^r}b_{\n} \T^{\n}\in\mathscr M_Y^{\hat\mu}[[\T]]$} is the formal series
\begin{equation}\label{form3.5}
\begin{aligned}
a(\T)\ast_{\H} b(\T):= \sum_{\n \in\mathbb N^r}a_{\n}\ast b_{\n}\T^{\n}
%\Big(\text{resp.}\quad a(\T)\ast_{i,\H} b(\T):= &\sum_{\n \in\mathbb N^r}a_{\n}\ast_i b_{\n}\T^{\n},\ i\in\{0,1\}\Big)
\end{aligned}
\end{equation}
in $\mathscr M_{X\times Y}^{\hat\mu}[[\T]]$. The associativity of the Hadamard product $\ast_{\H}$ is obtained from that of the convolution product $\ast$. Similarly to \cite[Lemma 7.6]{Loo}, the $\ast$-product is anti-compatible with the Hadamard product $\ast_{\H}$-product via the morphism $\lim_{T_1=\cdots=T_r\to\infty}$. Namely, for $r=1$ for instance, if $a(T)$ is in $\mathscr M_X^{\hat\mu}[[T]]_{\sr}$ and $b(T)$ is in $\mathscr M_Y^{\hat\mu}[[T]]_{\sr}$, then $a(T)\ast_{\H} b(T)$ is in $\mathscr M_{X\times Y}^{\hat\mu}[[T]]_{\sr}$, and moreover,
\begin{align}\label{form3.7}
\lim_{T\to\infty}\left(a(T)\ast_{\H} b(T)\right)=-\Big(\lim_{T\to\infty}a(T)\Big) \ast \Big(\lim_{T\to\infty}b(T)\Big).
\end{align}

The external convolution product can be extended to the following
\begin{align*}%\label{form3.5}
\ast: \MM_X^{\hat{\mu}}\times \MM_Y^{\hat{\mu}}\to \MM_{X\times Y}^{\hat{\mu}},
\end{align*}
which remains the properties mentioned previously.

%---------------

\subsection{The $\boxast$-product of integrable univariate series}\label{keyproduct}
Let $X$ and $Y$ be two algebraic $k$-varieties, and let $T$ and $U$ be univariates. In this paragraph, we introduce a new product of two integrable series $a(T)\in \MM_X^{\hat\mu}[[T]]_{\inte}$ and $b(U)\in \MM_Y^{\hat\mu}[[U]]_{\inte}$, which is an element of $\MM_X^{\hat\mu}[[T,U]]_{\inte}$ and commutes with the morphism $\lim_{T=U\to\infty}$.  

We use the augmentation map $\mathscr M_X^{\hat\mu}\to \mathscr M_X$ defined in \cite[Section 5]{Loo}, with remark that in the context of $\hat\mu$ the characteristic zero ground field is not necessarily algebraically closed. There is a more effective way to obtain a generalized augmentation map $\mathscr M_X^{\hat\mu^r}\to \mathscr M_X^{\hat\mu^{r-1}}$ using \cite[Proposition 2.6, Section 3.10]{GLM1}. The image of an element $z\in \mathscr M_X^{\hat\mu}$ under the augmentation map will be denoted by $z'$.

\begin{definition}\label{2-product}
The {\it $\boxast$-product of the series $a(T)=\sum_{n\geq 1}a_nT^n$ and $b(U)=\sum_{m\geq 1}b_mU^m$} is defined as follows 
\begin{align*}
a(T)\boxast b(U):=\sum_{n,m\geq 1}c_{n,m}T^nU^m \ \in \MM_{X\times Y}^{\hat\mu}[[T,U]],
\end{align*}
where $c_{n,m}$ equals
\begin{equation*}
\begin{cases}
(\Lbb-1)\sum_{l>m}a_n\times b'_l & \text{if}\ n<m,\\
(\Lbb-1)\sum_{l>n}a'_l\times b_m & \text{if}\ n>m,\\
-a_n\ast b_n+\sum_{l\leq n}\mathbb L^{l-n}a_l\ast_0 b_l+(\Lbb-1)\sum_{l>n}(a_n\times b'_l+a'_l\times b_n) & \text{if}\ n=m.
\end{cases}
\end{equation*}
\end{definition}

Remark that the integrability of $a(T)$ and $b(U)$ implies that $a(T)\boxast b(U)$ is well defined. Indeed, since $a(T)$ and $b(U)$ are integrable, they are of finite mass, a condition guarantees that $\sum_{l>n}a_l$ and $\sum_{l>m}b_l$ make sense and belong to $\MM_X^{\hat\mu}$ and $\MM_Y^{\hat\mu}$, respectively.

\begin{theorem}\label{commutativity}
The $\boxast$-product preserves the integrability and commutes with the limit of integrable series. More precisely, if $a(T)$ is in $\MM_X^{\hat\mu}[[T]]_{\inte}$ and $b(U)$ is in $\MM_Y^{\hat\mu}[[U]]_{\inte}$, then $a(T)\boxast b(U)$ is in $\MM_{X\times Y}^{\hat\mu}[[T,U]]_{\inte}$, and 
$$\lim_{T=U\to \infty} \left(a(T)\boxast b(U)\right)=\lim_{T\to \infty} a(T) \ast \lim_{U\to \infty}b(U).$$
\end{theorem}

\begin{proof}
The first statement that $a(T)\boxast b(U)$ is in $\MM_{X\times Y}^{\hat\mu}[[T,U]]_{\inte}$ if $a(T)$ is in $\MM_X^{\hat\mu}[[T]]_{\inte}$ and $b(U)$ is in $\MM_Y^{\hat\mu}[[U]]_{\inte}$ will be proved in the general case for several variables in Theorem \ref{preservingInt}. 

Let us prove the second one. Write the series $a(T)$, $b(U)$ and $a(T)\boxast b(U)$ as $\sum_{n\geq 1}a_n T^n$, $\sum_{m\geq 1}b_m U^m$ and $\sum_{n,m}c_{n,m}T^nU^m$, respectively. Take $T=U$ in $a(T)\boxast b(U)$ so that the resulting series can be written as 
$$\sum_{n,m} c_{n,m} T^{n+m}=A_1+A_2+(\mathbb{L}-1)(B_1+B_2),$$
where, by definition, 
\begin{gather*}
A_1=-\sum_{n\geq 1}a_n\ast b_n T^{2n},\quad A_2=\sum_{n\geq 1}\left(\sum_{l\leq n}\mathbb{L}^{l-n} a_l\ast_0 b_l\right)T^{2n},\\
B_1=\sum_{1\leq n\leq m}\left(a_n\times \sum_{l>m}b'_l\right)T^{n+m},\quad B_2=\sum_{1\leq m\leq n}\left(\sum_{l>n} a'_l\times b_m\right)T^{n+m}.
\end{gather*}
Here the integrability of $a(T)$ and $b(U)$ implies that $\sum_{l>n} a'_l$ converges in $\MM_X$ and $\sum_{l>m}b'_l$ converges in $\MM_Y$. The first limit is computed to be 
$$\lim_{T\to \infty} A_1=\lim_{T\to \infty} a(T^2) \ast \lim_{T\to \infty}b(T^2)=\lim_{T\to \infty} a(T) \ast \lim_{T\to \infty}b(T)$$ 
by means of (\ref{form3.7}). It is quite easy to obtain that  
\begin{align*}
\lim_{T\to \infty} A_2&=\lim_{T\to \infty} \sum_{n\geq 1}\left(\sum_{l=1}^{n}\mathbb (\Lbb^{-1})^{n-l} a_l\ast_0 b_l\right)T^{2n}\\
&=\left(\lim_{T\to \infty}\sum_{n\geq 0}\mathbb L^{-n} T^{2n}\right)\cdot\left(\lim_{T\to\infty}\sum_{n\geq 1} a_n\ast_0 b_n T^{2n}\right)
\end{align*}
which vanishes in $\MM_{X\times Y}^{\hat\mu}$, since $\lim_{T\to \infty}\sum_{n\geq 0}\mathbb L^{-n} T^{2n}$ vanishes in $\MM_{X\times Y}^{\hat\mu}$. The limits of $B_1$ and $B_2$ require more computations. It is verified in $\MM_{X\times Y}^{\hat\mu}[[T]]_{\inte}$ that
\begin{align*}
B_1&=\sum_{1\leq n\leq m}\left(a_n\times b'(1)\right)T^{n+m}-\sum_{1\leq n\leq m}\left(a_n\times \sum_{1\leq l\leq m}b'_l\right)T^{n+m}\\
&=\sum_{n\geq 1}  a_n \times b'(1)\sum_{1\leq m\leq n}T^{n+m} -\sum_{1\leq n\leq m}\left(a_n\times \sum_{1\leq l\leq m}b'_l\right)T^{n+m}\\
&= \frac{a(T^2)}{1-T}\times b'(1)-\sum_{1\leq n\leq m}\left(a_n\times \sum_{1\leq l\leq m}b'_l\right)T^{n+m},
%\frac{a(T^2)}{1-T}\times_{\H}\left(\frac{b(1)}{1-T}\sum_{n\geq 1}T^{2n}\right)
\end{align*}
and, similarly, that
\begin{align*}
B_2=\left(\frac{a'(1)}{1-T}\sum_{n\geq 1}T^{2n}\right)\times_{\H}\frac{b(T^2)}{1-T}-\sum_{1\leq m\leq n}\left(\sum_{l\leq n}a'_l\times b_m\right)T^{n+m},
\end{align*}
where $a'(T):=\sum_{n\geq 1}a'_nT^n$ and $b'(U):=\sum_{m\geq 1}b'_mU^m$. Note that we can extend naturally the augmentation map to a map $\mathscr M_{X\times Y}^{\hat\mu^2}\to \mathscr M_{X\times Y}^{\hat\mu}$, from which, for every $m$ and $n$, the elements $(\Lbb-1)a_n\times b'_m$ and $(\Lbb-1)a'_n\times b_m$ in $\mathscr M_{X\times Y}^{\hat\mu}$ coincide, since both are the image of $(\Lbb-1)a_n\times b_m$ in $\mathscr M_{X\times Y}^{{\hat\mu}^2}$. We may also refer to \cite[Proposition 2.6, Section 3.10]{GLM1} to see this fact. In $\MM_{X\times Y}^{\hat\mu}$, because $\Lbb-1$ is invertible, the identity $a_n\times b'_m=a'_n\times b_m$ holds. Thus, for each $\kappa\geq 1$, by combinatoric computation, we obtain the following identity in $\MM_{X\times Y}^{\hat\mu}$:
\begin{align*}
\sum_{n+m=\kappa}&\left(\sum_{1\leq n,j\leq m} a_n\times b'_j +\sum_{1\leq m,i\leq n} a'_i\times b_m\right)\\
&=\sum_{n+m=\kappa}\left(\sum_{1\leq n,j\leq m} a_n\times b'_j +\sum_{1\leq m,i\leq n} a_i\times b'_m\right)\\
&=\sum_{n+m=\kappa}\sum_{j\leq m} a_n\times b'_j+\sum_{i\leq \lfloor \frac{\kappa} 2\rfloor}a_i\times \sum_{j\leq \lfloor \frac{\kappa} 2\rfloor}b'_j,
\end{align*}
where $\lfloor \frac{\kappa} 2\rfloor$ is the integer part of $\frac{\kappa} 2$. It implies that, in $\MM_{X\times Y}^{\hat\mu}[[T]]_{\inte}$,
\begin{align*}
B_1+B_2&= \frac{a(T^2)}{1-T}\times b'(1)+a'(1)\times \frac{b(T^2)}{1-T}-\sum_{l\geq 1}a_l T^l\times \sum_{l\geq 1}\Big(\sum_{m\leq l}b'_m\Big)T^l\\
&\quad\quad -(1+T) \sum_{l\geq 1}\Big(\sum_{n\leq l}a_n\Big)T^{2l}\times_{\H} \sum_{l\geq 1}\Big(\sum_{m\leq l}b'_m\Big)T^{2l}\\
&= \frac{a(T^2)}{1-T}\times b'(1)+a'(1)\times \frac{b(T^2)}{1-T}-a(T) \times \frac{b'(T)}{1-T}\\
&\quad\quad -(1+T) \left(\frac{a(T^2)}{1-T^2} \times_{\H} \frac{b'(T^2)}{1-T^2}\right),
\end{align*}
since 
$$\sum_{\kappa\geq 1}\Big(\sum_{i\leq \lfloor \frac{\kappa} 2\rfloor}a_i\times \sum_{j\leq \lfloor \frac{\kappa} 2\rfloor}b'_j\Big)T^{\kappa}=(1+T) \sum_{l\geq 1}\Big(\sum_{n\leq l}a_n\times \sum_{m\leq l}b'_m\Big)T^{2l}.$$
Here, for any two series $\alpha(T)\in \MM_X^{\hat\mu}[[T]]$ and $\beta(T)\in \MM_Y^{\hat\mu}[[T]]$, by $\alpha(T)\times \beta(T)$ we mean the usual product of formal series in which the multiplication for the coefficients uses the external product $\times$. Now it is easy to obtain the vanishing of $\lim_{T\to \infty}(B_1+B_2)$ in $\MM_{X\times Y}^{\hat\mu}$, and the theorem is proved.
\end{proof}

%*************************************************************
\section{A motivic analogue of the Euler reflexion formula}
\subsection{Main theorem}

In this paragraph we state and prove an analogue of the Euler reflexion formula for motivic zeta functions, the most important result of the present article.

\begin{theorem}\label{thm21}
Let $X$ and $Y$ be smooth algebraic $k$-varieties, let $f$ and $g$ be regular functions on $X$ and $Y$ with the zero loci $X_0$ and $Y_0$, respectively. Define a function $f\oplus g$ on $X\times Y$ by $f\oplus g(x,y)=f(x)+g(y)$. Let $\iota$ be the inclusion of $X_0\times Y_0$ in $X\times Y$. Then the following identity
$$\zeta_{f}(T)\boxast \zeta_{g}(U)=\zeta_{f,g}(T,U)+\zeta_{g,f}(U,T)+\iota^*\zeta_{f\oplus g}(TU)$$
holds in $\MM_{X_0\times Y_0}^{\hat\mu}[[T,U]]$. It is called the {\rm motivic Euler reflexion formula} for $(f,g)$.
\end{theorem}

\begin{proof}
Let $d_1$ and $d_2$ be the pure $k$-dimensions of $X$ and $Y$, respectively, and let $d:=d_1+d_2$. For brevity of notation, we write $a_n$ for $[\mathscr X_n(f)]\mathbb L^{-nd_1}$ in $\mathscr M_{X_0}^{\hat\mu}$ and $b_n$ for $[\mathscr X_n(g)]\mathbb L^{-nd_2}$ in $\mathscr M_{Y_0}^{\hat\mu}$, we also ignore writing arrows to base for relative objects when they are clearly understood, e.g., let $[\mathscr X_n(f)]$ simply stand for $[\mathscr X_n(f)\to X_0]$. The motivic zeta functions of $f$ and $g$ can be rewritten as follows
$$\zeta_{f}(T)=\sum_{n\geq 1}a_nT^n \in \mathscr M_{X_0}^{\hat\mu}[[T]] \quad\text{and}\quad \zeta_{g}(U)=\sum_{n\geq 1}b_nU^n\in \mathscr M_{Y_0}^{\hat\mu}[[U]].$$

Let us consider the coefficients of the series $\iota^*\zeta_{f\oplus g}(TU)$. For $n\in\mathbb N_{>0}$, we have 
\begin{align*}
[\iota^*\mathscr X_n(f\oplus g)]=\left[\left\{(\varphi,\psi)\in \mathscr L_n(X\times Y)\mid f(\varphi)+ g(\psi)= t^n\mod t^{n+1}\right\}\right]
\end{align*}
that equals the sum $A_1^{(n)}+A_2^{(n)}+A_3^{(n)}$, where $A_1^{(n)}$, $A_2^{(n)}$ and $A_3^{(n)}$ are given by the expressions
\begin{align*}
A_1^{(n)}&=\left[\left\{(\varphi,\psi)\in \mathscr L_n(X\times Y)\ \begin{array}{|l} f(\varphi)+ g(\psi)= t^n\mod t^{n+1}\\ \ord f(\varphi)=\ord g(\psi)=n \end{array}\right\}\right],\\
A_2^{(n)}&=\left[\left\{(\varphi,\psi)\in \mathscr L_n(X\times Y)\ \begin{array}{|l} f(\varphi)+ g(\psi)= t^n\mod t^{n+1}\\ \ord f(\varphi)\neq \ord g(\psi)\end{array}\right\}\right],\\
A_3^{(n)}&=\sum_{1\leq l<n} \left[\left\{(\varphi,\psi)\in \mathscr L_n(X\times Y)\ \begin{array}{|l} f(\varphi)+ g(\psi)= t^n\mod t^{n+1}\\ \ord f(\varphi)= \ord g(\psi)=l\end{array}\right\}\right].
\end{align*}
It is useful for the rest of the proof to introduce another notation, $B_n$, so that
\begin{align*}
B_n=(\Lbb-1)\left[\left\{(\varphi,\psi)\in \mathscr L_n(X\times Y)\ \begin{array}{|l} f(\varphi)= t^n\mod t^{n+1}\\ g(\psi)= -t^n\mod t^{n+1}\end{array}\right\}\right].
\end{align*}

\begin{lemma}\label{lem31}
The identities $a_n\ast_1 b_n=A_1^{(n)}\mathbb L^{-nd}$ and $a_n\ast_0 b_n=B_n\mathbb L^{-nd}$ hold in $\mathscr M_{X_0\times Y_0}^{\hat\mu}$.
\end{lemma}

\begin{proof}[Proof of Lemma \ref{lem31}]
We shall prove the first identity, that $a_n\ast_1 b_n=A_1^{(n)}\mathbb L^{-nd}$, proving the second one can be done in the same way. The mapping from the $k$-variety $\mathscr X_n(f)\times \mathscr X_n(g)\times F_1^n$ toward the $k$-variety
$$E:=\left\{(\varphi,\psi)\in \mathscr L_n(X\times Y)\
\begin{array}{|l}
\ord f(\varphi)=\ord g(\psi)=n\\
f(\varphi)+g(\psi)= t^n\mod t^{n+1}
\end{array}
\right\}
$$
that sends $(\varphi(t),\psi(t); \xi,\eta)$ to $(\varphi(\xi t),\psi(\eta t))$ gives rise to a morphism $\theta$ of $(X_0\times Y_0)$-varieties 
$$\mathscr X_n(f)\times \mathscr X_n(g)\times^{\mu_n\times \mu_n} F_1^n\to E.$$ 
It is clear that the source and the target of $\theta$ are endowed with the natural action of $\mu_n$, and that $\theta$ is a $\mu_n$-equivariant isomorphism. The desired identity $a_n\ast_1 b_n=A_1^{(n)}\mathbb L^{-nd}$ is now proved. The reader may also find in the proof of Lemma 5.2 in \cite{Thuong2} to obtain more detailed arguments.
\end{proof}

\begin{lemma}\label{lem32}
The identity $(\mathbb L-1)\sum_{l>n}\left(a_n \times b'_l+a'_l \times b_n\right)=A_2^{(n)}\mathbb L^{-nd}$ holds in $\mathscr M_{X_0\times Y_0}^{\hat\mu}$.
\end{lemma}

%Proving the previous lemma is quite easy, possibly done by direct computation, so we leave this work to the reader.
\begin{proof}[Proof of Lemma \ref{lem32}]
Note that the condition $\ord f(\varphi)\neq \ord g(\psi)$ in the definition of $A_2^{(n)}$ may be presented as 
$$\left(\ord f(\varphi)< \ord g(\psi)\right) \vee \left(\ord f(\varphi) > \ord g(\psi)\right),$$
so we can write $A_2^{(n)}$ as follows
\begin{align*}
A_2^{(n)}&=\left[\left\{\varphi\in \mathscr L_n(X)\mid f(\varphi)= t^n\mod t^{n+1}\right\}\right]\times \left[\left\{\psi\in \mathscr L_n(Y)\mid \ord g(\psi)>n\right\}\right]\\
&\quad\quad  + \left[\left\{\varphi\in \mathscr L_n(X)\mid \ord f(\varphi)>n\right\}\right]\times \left[\left\{\psi\in \mathscr L_n(Y)\mid g(\psi)=t^n\mod t^{n+1}\right\}\right].
\end{align*}
Let us denote by $D$ the constructible subset $\left\{\varphi\in \mathscr L_n(X)\mid \ord f(\varphi)>n\right\}$ of $\mathscr L_n(X)$. Then $\mu(\pi_n^{-1}(D))=[D]\Lbb^{-d_1}$, with $\mu$ being the motivic measure. Putting
$$D_l:=\left\{\varphi\in \mathscr L(X)\mid \ord f(\varphi)=l\right\},$$
for any $l>n$, we get $\pi_n^{-1}(D)=\bigcup_{l>n}D_l$, and, by $\sigma$-additivity of $\mu$, we have 
$$[D]=\Lbb^{d_1}\mu(\pi_n^{-1}(D))=\Lbb^{d_1}\sum_{l>n}\mu(D_l)=\sum_{l>n}\left[\pi_l(D_l)\right]\Lbb^{-ld_1}.$$

Since $\left\{\varphi\in \mathscr L_l(X)\mid \ord f(\varphi)=l\right\}$ is isomorphic as an algebraic $X_0$-variety to the quotient of $\mathscr X_l(f)\times\mathbb G_{m,k}$ by the $\mu_l$-action given by $\xi\cdot (\varphi,\lambda):=(\xi\varphi,\xi^{-1}\lambda)$, we have $\left[\pi_l(D_l)\right]\Lbb^{-ld_1}=(\Lbb-1)a'_l$, thus we get 
$$[D]=(\Lbb-1)\sum_{l>n}a'_l,$$ 
and in the same way, $\left[\left\{\psi\in \mathscr L_n(Y)\mid \ord g(\psi)>n\right\}\right]$ is equal to $(\Lbb-1)\sum_{l>n}b'_l$. The lemma is then proved. 
\end{proof}

\begin{lemma}\label{lem33}
The equality $\sum_{l<n}a_l\ast_0 b_l\mathbb L^{l-n}=A_3^{(n)}\mathbb L^{-nd}$ holds in $\MM_{X_0\times Y_0}^{\hat\mu}$.
\end{lemma}

\begin{proof}[Proof of Lemma \ref{lem33}]
For any $l<n$, let us consider the $k$-varieties 
\begin{align*}
U_l&=\left\{(\varphi,\psi)\in \mathscr L_n(X\times Y)\ \begin{array}{|l} f(\varphi)+ g(\psi)= t^n\mod t^{n+1}\\ \ord f(\varphi)= \ord g(\psi)=l\end{array}\right\},\\
\widetilde U_l&=\left\{(\varphi,\psi)\in \mathscr L_n(X\times Y)\ \begin{array}{|l} f(\varphi)+ g(\psi)= t^n\mod t^{n+1}\\ f(\varphi)=t^l\mod t^{l+1}\\ g(\psi)=-t^l\mod t^{l+1}\end{array}\right\},
\end{align*}
and
\begin{align*}
W_l&=\left\{(\varphi,\psi)\in \mathscr L_n(X\times Y)\ \begin{array}{|l} \ord\left(f(\varphi)+g(\psi)\right)=n\\ f(\varphi)=t^l\mod t^{l+1}\\ g(\psi)=-t^l\mod t^{l+1}\end{array}\right\},
\end{align*}
which admit evidently the natural action of $\mu_l$. Here, the class of $U_l$ is nothing else than the $l$-th term of the sum $A_3^{(n)}$. Again as in the proof of Lemma \ref{lem32}, since $U_l$ is isomorphic as a $(X_0\times Y_0)$-variety to the quotient of $\widetilde U_l\times\mathbb G_{m,k}$ by the $\mu_l$-action given by $\xi\cdot (\varphi,\psi,\lambda):=(\xi\varphi,\xi\psi,\xi^{-1}\lambda)$, we get $[U_l]=(\Lbb-1)[\widetilde U_l]'$ in $\mathscr M_{X_0\times Y_0}^{\hat\mu}$. Similarly, $[W_l]=(\Lbb-1)[\widetilde U_l]'$ in $\mathscr M_{X_0\times Y_0}^{\hat\mu}$. Hence $[U_l]=[W_l]$ in $\mathscr M_{X_0\times Y_0}^{\hat\mu}$. 

Now we write $[W_l]=[W_l^{\geq n}]-[W_l^{\geq n+1}]\Lbb^{-d}$, where 
$$W_l^{\geq n}=\left\{(\varphi,\psi)\in \mathscr L_n(X\times Y)\ \begin{array}{|l} \ord\left(f(\varphi)+g(\psi)\right)\geq n\\ f(\varphi)=t^l\mod t^{l+1}\\ g(\psi)=-t^l\mod t^{l+1}\end{array}\right\}.$$
Put 
\begin{align*}
E_{n,l}&=\left\{(\varphi,\psi)\in \mathscr L_{n}(X\times Y)\ \begin{array}{|l}  f(\varphi)=t^l\mod t^{l+1}\\ g(\psi)=-t^l\mod t^{l+1} \end{array}\right\}\\
A&=\left\{\tau\in \mathscr L_{n-l}(\mathbb{A}_k^1)\mid \tau=t+\alpha_2t^2+\cdots+\alpha_{n-l}t^{n-l}\right\}.
\end{align*}
The $(X_0\times Y_0)$-morphism $W_l^{\geq n}\times A \to E_{n,l}$ sending $(\varphi,\psi,\tau)$ to $(\varphi\circ\tau,\psi\circ\tau)$ is an isomorphism, from which $[W^{\geq n}_l]=[E_{n,l}]\Lbb^{l+1-n}$. Since $[E_{n,l}]=B_l(\Lbb-1)^{-1}\Lbb^{(n-l)d}$, it follows from Lemma \ref{lem31} that $[E_{n,l}]=a_l\ast_0 b_l(\Lbb-1)^{-1}\Lbb^{nd}$, therefore 
$$[W^{\geq n}_l]=a_l\ast_0 b_l(\Lbb-1)^{-1}\Lbb^{nd+l+1-n}.$$ 
Consequently, 
$$[W_l]=[W_l^{\geq n}]-[W_l^{\geq n+1}]\Lbb^{-d}=a_l\ast_0 b_l\Lbb^{nd+l-n}.$$
Then we get $A_3^{(n)}\mathbb L^{-nd}=\sum_{l<n}[W_l]\mathbb L^{-nd}=\sum_{l<n}a_l\ast_0 b_l\Lbb^{l-n}$ as desired.
\end{proof}

Let us continue the proof of Theorem \ref{thm21}. Using Lemmas \ref{lem31}, \ref{lem32} and \ref{lem33} gives the coefficient of $T^nU^n$ in $\iota^*\zeta_{f\oplus g}(TU)$, also the coefficient of $T^nU^n$ in the right hand side of the Euler reflexion formula, as follows
\begin{align*}
[\iota^*\mathscr X_n(f\oplus g)]\Lbb^{-nd}&=a_n\ast_1 b_n+ \sum_{l<n} \mathbb L^{l-n}a_l\ast_0 b_l+ (\mathbb L-1)\sum_{l>n}\left(a_n \times b'_l+a'_l \times b_n\right)\\
&=-a_n\ast b_n+ \sum_{l\leq n}\mathbb L^{l-n}a_l\ast_0 b_l+ (\mathbb L-1)\sum_{l>n}\left(a_n \times b'_l+a'_l \times b_n\right).
\end{align*}
This quantity agrees with the coefficient of $T^nU^n$ in the left hand side, according to the $\boxast$-product of the motivic zeta functions $\zeta_{f}(T)$ and $\zeta_{g}(U)$ (see Section \ref{keyproduct}).  

On the other hand, for $n<m$, the coefficient of $T^nU^m$ in the right hand side of the Euler reflexion formula is nothing else than $\left[\mathscr{D}_{n,m}(f,g)\right]\Lbb^{-(n+m)d}$, which equals
$$[\mathscr X_n(f)]\mathbb L^{-nd_1}\times \sum_{l>m}\left[\left\{(\psi\in\mathscr{L}_l(Y)\mid \ord g(\psi)=l\right\}\right]\mathbb L^{-ld_2}=(\Lbb-1)\sum_{l>m}a_n\times b'_l,$$
definitely coinciding the coefficient of $T^nU^m$ in the left hand side of the Euler reflexion formula. For the detail in proving these identities, see the proof of Lemma \ref{lem32}. The previous arguments obviously run for the case $n>m$, and Theorem \ref{thm21} is now proved.
\end{proof}

%*************************
\subsection{Motivic multiple nearby cycles and the motivic Thom-Sebastiani theorem}
Let $X$, $Y$, $f$ and $g$ be as in Theorem \ref{thm21}. Let us now compute the motivic multiple zeta functions $\mathscr S_{f,g}$ and $\mathscr S_{g,f}$, which are the limit of the series $-\zeta_{f,g}(T,T)$ and $-\zeta_{g,f}(T,T)$, respectively. Afterward, together with the commuting of $\boxast$-product and $\lim_{T\to \infty}$, and the motivic Euler reflexion formula, we deduce the motivic Thom-Sebastiani theorem.

\begin{proposition}\label{lm23}
The identities $\mathscr S_{f,g}=-\mathscr S_f\times [Y_0]$ and $\mathscr S_{g,f}=-[X_0]\times\mathscr S_g$ hold in $\mathscr M_{X_0\times Y_0}^{\hat\mu}$.
\end{proposition}
\begin{proof}
It suffices to check for the first identity. As in the proof of Theorem \ref{thm21}, for brevity of notation, let $a_n$ and $b_n$ stand for $[\mathscr X_n(f)]\mathbb L^{-nd_1}$ and $[\mathscr X_n(g)]\mathbb L^{-nd_2}$, respectively. By definition, 
$$\mathscr S_{f,g}=-\lim_{T\to\infty}\zeta_{f,g}(T,T),$$ 
we get the following
\begin{align*}
-\mathscr S_{f,g}&=(\mathbb{L}-1)\lim_{T\to\infty}\sum_{1\leq n< m} a_n\times\sum_{l> m} b'_l T^{n+m}\\
&=(\mathbb{L}-1)\lim_{T\to\infty}\sum_{1\leq n< m} a_n\times \sum_{l\geq 1} b'_l T^{n+m}-(\mathbb{L}-1)\lim_{T\to\infty}\sum_{1\leq n< m} a_n\times \sum_{l\leq m} b'_l T^{n+m}\\
&=\lim_{T\to\infty}\sum_{1\leq n< m} a_n T^{n+m}\times [Y_0]-(\mathbb{L}-1)\lim_{T\to\infty}\sum_{1\leq n< m} a_n \times\sum_{l\leq m} b'_l T^{n+m}\\
&=\lim_{T\to\infty}\sum_{n\geq 1} a_n \frac{T^{2n+1}}{1-T}\times [Y_0]-(\mathbb{L}-1)\lim_{T\to\infty}\sum_{1\leq n<m}\sum_{l\leq m} a_n \times b'_l T^{n+m}\\
&=\lim_{T\to\infty}\frac{T\zeta_{f}(T^2)}{1-T}\times [Y_0]+(\mathbb{L}-1)\lim_{T\to\infty} Z^{C,\ell}_{f,\id,g}\\
&=\mathscr S_{f}\times [Y_0]+(\mathbb{L}-1)\mathscr S^{C,\ell}_{f,\id,g},
\end{align*}
where $C$ is the rational polyhedral convex cone 
$$\left\{(n,l,m)\in \mathbb{N}^3\mid 1\leq n<m,1\leq l\leq m\right\},$$ 
$\ell(n,m,l)=n+m$, for $(n,m,l)\in\mathbb R^3$, and $\id$ is the identity morphism on $\mathbb{A}_k^1$. According to \cite[Section 2.9]{GLM2}, in fact, $\mathscr S^{C,\ell}_{f,\id,g}$ is independent of the choice of $\ell$ provided $\ell$ is linear on $\mathbb R^3$ and positive on the closure of $C$ in $\mathbb R^3$ outside the origin. By this, we may replace $\ell$ by $\ell'$ defined by $\ell'(n,m,l)=m$ to get $Z^{C,\ell'}_{f,id,g}(T)$ so that $Z^{C,\ell'}_{f,id,g}(T)$ has the same limit $\lim_{T\to\infty}$ as $Z^{C,\ell}_{f,id,g}(T)$. More precisely,
$$
-\mathscr S^{C,\ell}_{f,\id,g}=-\mathscr S^{C,\ell'}_{f,\id,g}=\lim_{T\to\infty} Z^{C,\ell'}_{f,\id,g}=\lim_{T\to\infty}\sum_{1\leq n<m}\sum_{1\leq l\leq m} a_n \times b_l T^{m}.
$$
By applying Lemma 7.6 of \cite{Loo} to the external product $\times$, which was already recalled in Section \ref{convolution}, together with the previous identity, we obtain a formula for $-\mathscr S^{C,\ell}_{f,\id,g}$ as follows
\begin{equation*}
\begin{aligned}
-\mathscr S^{C,\ell}_{f,\id,g}&=\lim_{T\to\infty}\sum_{m\geq 1}\Big(\sum_{n< m} a_n \times \sum_{l\leq m} b_l\Big)T^m\\
&=-\lim_{T\to\infty}\sum_{m\geq 1} \Big(\sum_{n< m } a_n\Big)T^{m} \times \lim_{T\to\infty}\sum_{m\geq 1} \Big(\sum_{l\leq m} b_l\Big) T^{m}\\
&=-\lim_{T\to\infty}\sum_{m\geq 1} a_m T^{m+1} \times \lim_{T\to\infty}\sum_{m\geq 1} b_m T^m\\
&=\lim_{T\to\infty}\sum_{m\geq 1} a_m \frac{T^{m+1}}{1-T}\times\lim_{T\to\infty}\sum_{m\geq 1} b_m \frac{T^{m}}{1-T},
\end{aligned}
\end{equation*}
which vanishes because of the vanishing of the second factor of the last expression, completing the proof of Proposition \ref{lm23}.
\end{proof}

\begin{theorem}[Motivic Thom-Sebastiani theorem]\label{coro21}
Using the assumption as in Theorem \ref{thm21}, the following identity
$$\iota^*\mathscr S_{f\oplus g}=-\mathscr S_{f}\ast \mathscr S_{g}+\mathscr S_f\times [Y_0]+[X_0]\times\mathscr S_g$$
holds in $\MM_{X_0\times Y_0}^{\hat\mu}$.
\end{theorem}
\begin{proof}
This is a direct consequence of Theorems \ref{thm21}, \ref{commutativity} and Proposition \ref{lm23}.
\end{proof}

%*********************************

\section{Generalization of $\boxast$-product and motivic Euler reflexion formula}
\subsection{Integrable series}
First of all, let us recall some basic results on integrability of formal series. We define 
$$\mathbb Z[\Lbb]_{\loc}:=\mathbb Z[\Lbb, \Lbb^{-1}, (1-\Lbb^n)^{-1}, n\geq 1].$$ 
Let $\mathscr M$ and $\mathscr N$ be $\mathbb Z[\Lbb]_{\loc}$-modules, and let $\mathscr M\otimes \mathscr N$ denote $\mathscr M\otimes_{\mathbb Z[\Lbb]_{\loc}} \mathscr N$ for short.

\begin{lemma}\label{lm51}
If $a(\T)\in \mathscr M[[\T]]_{\inte}$ and $b(\T)\in \mathscr N[[\T]]_{\ssr}$, then $a(\T)\otimes_{\H} b(\T)\in \mathscr M\otimes \mathscr N[[\T]]_{\inte}$.
\end{lemma}

\begin{proof}
Looijenga gave a similar statement for the univariate case in \cite[Lemma 7.6]{Loo}, which claims that the Hadamard product corresponding to tensor product on coefficients of two rational series is again rational. His arguments in fact still work in our situation. Moreover, there are methods more direct to prove this lemma, such as combinatorics or Cluckers-Loeser's computations for the constructible motivic functions in \cite[Section 4]{CL} together with the version with action in \cite{Thuong3}.
\end{proof}

\begin{lemma}\label{lm511}
Let $\mathscr M$ be $\mathbb Z[\Lbb]_{\loc}$-modules, and $\T$ and $\U$ separated multivariates. Then
$$\mathscr M[[\T]]_{\inte}[[\U]]_{\inte}\subset \mathscr M[[\T,\U]]_{\inte}\subset \mathscr M[[\T]]_{\inte}[[\U]],$$
where $\mathscr M[[\T]_{\inte}[[\U]]$ is the set of formal series in $\U$ over $\mathscr M[[\T]]_{\inte}$.
\end{lemma}

Proof of Lemma \ref{lm511} is elementary and left to the readers.

For any pair of formal series with some variables mixed, namely, $a(\T,\V)=\sum_{\n,\l}a_{\n,\l}\T^{\n}\V^{\l}$ in $\mathscr M[[{\T,\V}]]$ and $b(\U,\V)=\sum_{\m,\l}b_{\m,\l}\U^{\m}\V^{\l}$ in $\mathscr N[[\U,\V]]$, their {\it $\V$-Hadamard product} is an element of $\mathscr M\otimes \mathscr N[[{\T,\U,\V}]]$ given by
\begin{align}\label{newHadamard}
a(\T,\V)\otimes_{\H} b(\U,\V):=\sum_{\n,\m,\l}\a_{\n,\l}\otimes \b_{\m,\l}\T^{\n}\U^{\m} \V^{\l}.
\end{align}

\begin{lemma}\label{lm52}
If $a(\T,\V)$ is in $\mathscr M[[{\T,\V}]]_{\inte}$ and $b(\U,\V)$ is in $\mathscr N[[\U]]_{\inte}[[\V]]_{\ssr}$, then the $\V$-Hadamard product $a (\T,\V)\otimes_{\H} b(\U,\V)$ is in $\mathscr M\otimes \mathscr N[[\T,\U,\V]]_{\inte}$.
\end{lemma}
\begin{proof}
It is easy to see that the series $c(\T,\U,\V):=a(\T,\V)\otimes_{\H} b(\U,\V)$ can be presented as the Hadamard product of two elements of $\mathscr M\otimes \mathscr N[[{\T,\U,\V}]]$ as follows
\begin{align}\label{expressionC}
c(\T,\U,\V)=\frac{a (\T,\V)}{\prod (1-U_j)}\otimes_{\H} \frac{b (\U,\V)}{\prod (1-T_i)},
\end{align}
where $\prod (1-T_i):=(1-T_1)\cdots (1-T_r)$ and $\prod (1-U_j):=(1-U_1)\cdots (1-U_s)$. By setting $b(\U,\V)=\sum_{\m} b_{\m}(\V)\U^{\m}$, we may write the factors in (\ref{expressionC}) as follows
\begin{align*}
\frac{a (\T,\V)}{\prod (1-U_j)}&=\sum_{\m} a(\T,\V)\U^{\m} \in \mathscr M\otimes \mathscr N[[{\T,\V}]]_{\mathrm{int}}[[\U]]_{\ssr},\ \text{and}\\
 \frac{b (\U,\V)}{\prod (1-T_i)}&= \sum_{\m} \widetilde b_{\m}(\T,\V)\U^{\m}\in \mathscr M\otimes \mathscr N[[{\T,\V}]]_{\ssr}[[\U]]_{\inte},
\end{align*}
where 
$$\widetilde b_{\m}(\T,\V):=\frac{b_{\m}(\V)}{\prod (1-T_i)}.$$ 
This together with Lemma \ref{lm51} implies that $c(\T,\U,\V)\in \mathscr M\otimes \mathscr N[[\T,\V]]_{\ssr}[[\U]]_{\inte}$. Moreover, we have
$$c(\T,\U,\V)=\sum_{\m}a(\T,\V)\otimes_{\H}\widetilde b_{\m}(\T,\V)\U^{\m},$$
which belongs to $\mathscr M\otimes \mathscr N[[{\T,\V}]]_{\inte}[[\U]]$, by Lemma \ref{lm51}. It follows that $c(\T,\U,\V)$ is an element of $\mathscr M\otimes \mathscr N[[{\T,\V}]]_{\inte}[[\U]]_{\inte}$, hence an element of $\mathscr M\otimes \mathscr N[[\T,\U,\V]]_{\inte}$.
\end{proof}

Let $X_i$, $1\leq i\leq r$, be smooth algebraic $k$-varieties, and let $X:=X_1\times\cdots\times X_r$. As usual we use the multivariate $\T=(T_1,\dots,T_r)$. To each $1\leq i\leq r$ and formal series $a(\T)=\sum_{\n}a_{\n}\T^{\n}$ in $\mathscr M_X^{\hat\mu}[[\T]]$ associate a unique formal series $a_i(\T):=\sum_{\n}a_{\n}^{(i)}\T^{\n}$ in $\mathscr M_{X_i}^{\hat\mu}[[\T]]$ in such a way that $a^{(i)}_{\n}=(\pr_i)_!a_{\n}\in \mathscr M_{X_i}$, where $\pr_i$ is the natural projection of $X$ onto $X_i$. 

\begin{lemma}\label{lm53}
If the series $a(\T)$ is integrable, so are the series $a_{i}(\T)$ for $1\leq i\leq r$.
\end{lemma}

Proof of this lemma is straightforward. 

\begin{definition}
For any $r\in\mathbb N_{>0}$ and $1\leq i\leq r$, a $r$-tuple $\n=(n_1,\dots,n_r)\in\mathbb N^r$ is said to have {\it the $\Delta_{i,<}$-property} (resp. {\it the $\Delta_{<}$-property}), written as $\n\in\Delta_{i,<}$-property (resp. $\n\in\Delta_{<}$-property) or simply as $\n\in\Delta_{i,<}$ (resp. $\n\in\Delta_<$), if
$$n_1<\dots<n_i=n_{i+1}=\dots=n_r\quad \text{(resp. $n_1<\dots<n_r$)}.$$ 
\end{definition}
We denote by $\mathscr M_X^{i,<}[[\T]]$ (resp. $\mathscr M_X^<[[\T]]$) the subset of $\mathscr M_X^{\hat\mu}[[\T]]$ consisting of formal series of the form $\sum_{\n\in\Delta_{i,<}}a_{\n}\T^{\n}$ (resp. $\sum_{\n\in\Delta_<}a_{\n}\T^{\n}$). We also have an analogous definition for $\MM_X^{i,<}[[\T]]$ and $\MM_X^<[[\T]]$ as subset of $\MM_X^{\hat\mu}[[\T]]$. By definition, for any $a(\T)$ in $\mathscr M_X^{i,<}[[\T]]$ (resp. in $\MM_X^{i,<}[[\T]]$), there exists a series $\widetilde a(T_1,\dots,T_i)$ in $\mathscr M_X^<[[T_1,\dots,T_i]]$ (resp. in $\MM_X^<[[T_1,\dots,T_i]]$) such that 
$$a(\T)=\widetilde a(T_1,\dots,T_{i-1},T_i\cdots T_r).$$

Let us now introduce a new notion of ordered cells. For an increasing sequence of positive integers $0=r_0<r_1<\cdots<r_i=r$ we define the {\it basic ordered cell} $\Delta_{(r_0,\dots,r_i)}$ to be the set
$$\left\{(n_1,\dots,n_r)\in\mathbb N^r \mid n_{r_{j-1}+1}=\cdots=n_{r_j}\ \mathrm{and}\ n_{r_{j-1}}<n_{r_j},\ 2 \leq j\leq i\right\}.$$
A subset $\Delta$ of $\mathbb{N}^r$ is called an {\it ordered cell } if it is the image of a basic ordered cell $\Delta_{(r_0,\dots,r_i)}$ under a permutation map $\rho\colon \mathbb{N}^r\to \mathbb{N}^r$ that sends $(n_1,\ldots,n_r)$ to $(n_{\rho(1)},\dots,n_{\rho(r)})$. It is easy to see that $\mathbb N^r$ can be partitioned into all the ordered cells $\Delta$. This implies that any formal series $a(\T)\in \MM_X^{\hat\mu}[[\T]]$ can be uniquely decomposed as a finite sum of formal series 
\begin{align}\label{aDelta}
a(\T)=\sum_{\Delta} a_{\Delta}\left(\T\right)=\sum_{\Delta}a_{\Delta}^{<}\Big(\prod_{l=1}^{r_1}T_{\rho(l)},\dots,\prod_{l=r_{i-1}+1}^{r_i}T_{\rho(l)}\Big),
\end{align}
where $a_{\Delta}(\T):=\sum_{\n\in\Delta} a_{\n}\T^{\n} $ and $a_{\Delta}^{<}\in \MM_X^{<}[[T_1,\dots,T_i]]$ in viewing $X$ as 
$$\prod_{l=1}^{r_1} X_{\rho(l)}\times\cdots\times \prod_{l=r_{i-1}}^{r_i}X_{\rho(l)}.$$ 

\begin{lemma}\label{lm54}
If the series $a(\T)$ is integrable, so are the series $a_{\Delta}(\T)$ for all ordered cells $\Delta$.
\end{lemma}

\begin{remark}
Actually, in view of Cluckers-Loeser's theory on constructible motivic functions one can show that the lemma also works for any definable subset of $\mathbb{N}^r$, cf. \cite[Lemma 4.5.8]{CL}.
\end{remark}

\begin{proof}[Proof of Lemma \ref{lm54}]
It suffices to prove that $a_{\Delta}\left(\T\right)$ is integrable for $\Delta=\Delta_{(r_0,\dots,r_i)}$ being a basic ordered cell. We can check easily that
$$a_{\Delta}\left(\T\right)=\varepsilon(\T)\cdot_{\H} a(\T),$$
where, be definition,
$$\varepsilon(\T):= \frac{\prod_{j=2}^i\left(\prod_{l=r_{j-1}+1}^rT_l\right)}{\prod_{j=1}^i\left(1-\prod_{l=r_{j-1}+1}^rT_l\right)},$$
which is strongly rational. Then the present lemma follows from Lemma \ref{lm51}.
\end{proof}

%--------------------
We consider the morphism of $\MM_X^{\hat\mu}$-modules
$$\Phi\colon \MM_X^{<}[[\T]]\to \MM_X^{<}[[\T]]$$ 
given by
$$\Phi\Big(\sum_{\n} a_{\n} \T^{\n}\Big)=(\mathbb L-1)^{1-r}\sum_{\n}a^{(1)}_{\n}\times\prod_{i=2}^r(a^{(i)}_{\n-\e_i}-a^{(i)}_{\n})\T^{\n},$$
where $a^{(i)}_{\n}:=(\pr_i)_!a_{\n}\in \MM_{X_i}$, $\pr_i$ is the natural projection of $X$ onto $X_i$, and $\e_i$ is the $i$-th standard vector in $\mathbb{Z}^r$, $1\leq i\leq r$. Here by $\prod_{i=2}^r$ we mean temporarily the external products. It is clear that $\Phi$ can be extended to an endomorphism of $\MM_X^{\hat\mu}[[\T]]$,
\begin{align}\label{Phi5}
\Phi: \MM_X^{\hat\mu}[[\T]]\to \MM_X^{\hat\mu}[[\T]],
\end{align}
by linearity, namely,
$$\Phi(a(\T)):=\sum_{\Delta} \Phi\Big(a_{\Delta}\Big(\prod_{l=1}^{r_1}T_{\rho(l)},\ldots,\prod_{l=r_{i-1}+1}^{r_i}T_{\rho(l)}\Big)\Big),$$ 
in terms of the decomposition of $a(\T)\in \MM_X^{\hat\mu}[[\T]]$ into finitely many terms of the form (\ref{aDelta}). 

Now we work with the restriction of $\Phi$ to the sub-$\MM_X^{\hat\mu}$-module $\MM_X^{\hat\mu}[[\T]]_{\inte}$ of $\MM_X^{\hat\mu}[[\T]]$.

\begin{lemma}\label{lm55}
The restriction of $\Phi$ to $\MM_X^{\hat\mu}[[\T]]_{\inte}$ is an automorphism. 
\end{lemma}

\begin{proof}
Define the morphism $\Phi^{-1}: \MM_X^{<}[[\T]]_{\inte} \to \MM_X^{<}[[\T]]_{\inte}$ as follows
$$a(\T)=\sum_{\n\in\Delta_{<}} a_{\n} \T^{\n}\mapsto (\mathbb L-1)^{r-1}\sum_{\n\in\Delta_{<}}a^{(1)}_{\n}\times \prod_{i=2}^r \Big(\sum_{l>1}a^{(i)}_{\n+l\e_i}\Big)\T^{\n},$$
with $\prod_{i=2}^r$ being the external products at the moment. Let us show that $\Phi^{-1}(a(\T))$ is an integrable series. We first prove that, for any $2\leq i\leq r$, 
$$(\Lbb-1)\sum_{\n\in\Delta_{<}}\sum_{l>1}\a^{(i)}_{\n+l\e_i} \T^{\n}=\frac{\gamma_i(\T)}{1-T_i},$$
for some $\gamma_{i}(\T)\in \MM_{X_i}^{\hat\mu}[[\T]]_{\inte}$. Indeed, by setting $\hat{\n}_i:=\n-n_i\e_i$ and $\hat{\T}_i:=\T-(T_i-1)\e_i$, $2\leq i\leq r$, we have
\begin{align*}
(\Lbb-1)\sum_{\n\in\Delta_{<}}\sum_{l>1}\a^{(i)}_{\n+l\e_i} \T^{\n}&=(\Lbb-1)\sum_{l>1}\sum_{\hat{\n}_i\in\Delta_{<}}\a^{(i)}_{\hat{\n}_i+l\e_i}{\hat{\T}_i}^{\hat{\n}_i} \sum_{i\leq n_i<l}{T_i}^{n_i}\\
&=(\Lbb-1)\sum_{l>1}\sum_{\hat{\n}_i\in\Delta_{<}}\a^{(i)}_{\hat{\n}_i+l\e_i}{\hat{\T}_i}^{\hat{\n}_i} \frac{{T_i}^{i}-T_i^{l}}{1-T_i}\\
&=\frac{(\Lbb-1){T_i}^i}{1-T_i} a_i(\hat{\T}_i) -\frac{(\Lbb-1)a_i(\T)}{1-T_i} ,
\end{align*}
which has the form as desired. It therefore follows that
\begin{align*}
\Phi^{-1}(a(\T))=a_1(\T)\cdot_{\H}\frac{\gamma_2(\T)}{1-T_2}\cdot_{\H}\dots \cdot_{\H}\frac{\gamma_r(\T)}{1-T_r},
\end{align*}
which is obviously integrable due to Lemma \ref{lm51}. By the decomposition (\ref{aDelta}), the morphism $\Phi^{-1}$ may be extended to $\MM_X^{\hat\mu}[[\T]]_{\inte}$. It is easily checked that $\Phi\circ \Phi^{-1}=\Phi^{-1}\circ \Phi=\id_{\MM_X^{\hat\mu}[[\T]]_{\inte}}$. The lemma is thus proved. 
\end{proof}

\begin{corollary}\label{integrabilitymmzf}
Let $\f=(f_1,\dots,f_r)$ be an ordered family of regular functions on $X_1,\dots,X_r$. Then the multiple motivic zeta function $\zeta_{\f}(\T)$ is an integrable series, i.e., $\zeta_{\f}(\T)\in \MM_X^{\hat\mu}[[\T]]_{\inte}$.
\end{corollary}

\begin{proof}
Let $a(\T)= Z_{\f}^{\mathbb N_{>0}^r}(\T)$ be the motivic zeta function with respect to the trivial cone $\mathbb N_{>0}^r$ defined as in Definition \ref{MMZF}. Then we have
$$a(\T)=Z_{f_1}(T_1)\times_{\H}\cdots \times_{\H} Z_{f_r}(T_r),$$
it is therefore integrable due to Lemma \ref{lm52}. On the other hand, we deduce from Lemma \ref{lm54} that the series $a_{\Delta_{<}}(\T)=Z^{\Delta_{<}}_{\f}(\T)$ is integrable. Since the identity  $\Phi^{-1}\left(Z^{\Delta_{<}}_{\f}(\T)\right)=\zeta_{\f}(\T)$ holds in $\MM_X^{<}[[\T]]$, Lemma \ref{lm55} gives us the integrability of the series $\zeta_{\f}(\T)$.
\end{proof}

%-----------------

\subsection{Generalized $\boxast$-product}
Let $X_i$ and $Y_j$, $1\leq i\leq r$, $1\leq j\leq s$, be smooth algebraic $k$-varieties, and let 
\begin{align}\label{XY}
X:=X_1\times\cdots\times X_r\quad \text{and}\quad Y:=Y_1\times\cdots\times Y_s.
\end{align}
As usual we also use the multivariates $\T=(T_1,\dots,T_r)$ and $\U=(U_1,\dots,U_s)$. Now for tuples $\n=(n_1,\dots,n_r)$ and ${\bf m}=(m_1,\dots,m_s)$ having the $\Delta_{<}$-property, we let
$$I:=I_{\n,\m}:=\{(i,j)\in \mathbb N^2\mid n_i=m_j\},$$
and let $I_1$ (resp. $I_2$) be the image of $I$ under the projection on the first component (resp. the second component). Then, to define the $\boxast$-product of a series in $\MM_X^{\hat\mu}[[\T]]$ and a series in $\MM_Y^{\hat\mu}[[{\U}]]$ it suffices to define the $\boxast$-product of a series in $\MM_X^{<}[[\T]]$ and a series in $\MM_Y^{<}[[{\U}]]$.

\begin{definition}\label{n-product}
Let $a(\T)=\sum a_{\n}\T^{\n}$ and $b({\U})=\sum b_{\m}\U^{\m}$ be formal series in $\MM_X^{<}[[\T]]$ and $\MM_Y^{<}[[{\U}]]$, respectively. We define the product $a(\T)\boxast b(\U)$ in two steps as follows.
%\begin{itemize}

\medskip
(i) Put
$$a(\T)\boxast_0 b({\U}):=\sum_{\n\in \Delta_<,\m\in \Delta_<}c_{\n,\m}\T^{\n}{\U}^{\m},$$
where 
$$c_{\n,\m}=\prod_{i\not\in I_1}a^{(i)}_{\n}\times \prod_{j\not\in I_2}b^{(j)}_{\m}\times \prod_{(i,j)\in I} \widetilde{c}_{\n,\m}^{(i,j)},$$
and, for any $(i,j)\in I$, the quantity $\widetilde{c}_{\n,\m}^{(i,j)}$ is defined to be
$$-a^{(i)}_{\n}\ast b^{(j)}_{\m}+\sum_{0\leq l< n_i}\mathbb L^{-l}a^{(i)}_{\n-le_i}\ast_0 b^{(j)}_{\m-le_j}+(\Lbb-1)\sum_{l>0}\left(a^{(i)}_{\n}\times (b^{(j)}_{\m+le_j})'+(a^{(i)}_{\n+le_i})'\times b^{(j)}_{\m}\right)$$
with $z'$ the image of $z$ under the augmentation map.

(ii) Put
$$a(\T)\boxast b({\U}):=\Phi^{-1}\left(\Phi (a(\T))\boxast_0 \Phi(b({\U}))\right),$$
where $\Phi$ is defined previously in (\ref{Phi5}).
%\end{itemize}
\end{definition}

It is clear that the $\boxast$-product in Definition \ref{n-product} is well defined since $\Phi$ is well defined. Moreover, when reduced to the univariate case, i.e., $r=s=1$, this product is nothing else than the one defined in Definition \ref{2-product}.

\begin{theorem}\label{preservingInt}
With previous notation and hypotheses, if $a(\T)$ is in $\MM_X^{\hat\mu}[[\T]]_{\inte}$ and $b(\U)$ is in $\MM_Y^{\hat\mu}[[\U]]_{\inte}$, then $a(\T)\boxast b(\U)$ is in $\MM_{X\times Y}^{\hat\mu}[[\T,\U]]_{\inte}$. 
\end{theorem}

\begin{proof}
We first consider $a(\T)$ and $b(\U)$ in $\MM^{<}_X[[\T]]_{\inte}$ and $\MM^{<}_Y[[\U]]_{\inte}$, respectively. It follows from the proof of Lemma \ref{lm55} that 
$$(\mathbb L-1)\sum_{\n\in\Delta_{<}}\sum_{l>0}\a^{(i)}_{\n+l\e_i} \T^{\n}=\frac{\alpha_i(\T)}{1-T_i},$$
and that
$$(\mathbb L-1)\sum_{\m\in\Delta_{<}}\sum_{l>0}\b^{(j)}_{\m+l\e_j} \U^{\m}=\frac{\beta_j(\U)}{1-U_j},$$
for some integrable series $\alpha_{i}(\T)$ and $\beta_j(\U)$. Then, by simple computation, we deduce that
\begin{align*}
a(\T)\boxast_0 b(\U)=\prod_{i\not\in I_1}a_i(\T)\times_{\H} {\prod_{j\not\in I_2}} b_j(\U)\times_{\H} {\prod_{(i,j)\in I}} c_{ij}(\T,\U),
\end{align*}
where for each $(i,j)\in I$, the series $c_{ij}(\T,\U)$ is equal to
$$-a_i(\T_i)\ast_{\H} b_j(\U_j)+\frac{a_i(\T_i){\ast_0}_{\H} b_j(\U_j)}{1-\mathbb{L}^{-1}T_iU_j}+a_i(\T_i)
\times_{\H}\frac{\beta_j(\U_j)}{1-U_j}+\frac{\alpha_i(\T_i)}{1-T_i}\times_{\H} b_j(\U_j),$$
where 
$$\T_i:=(T_1,\dots,T_{i-1},T_iU_j,T_{i+1},\dots,T_r)$$ 
and 
$$\U_j:=(U_1,\dots,U_{j-1},T_iU_j,U_{j+1},\dots,U_s).$$ 
By using Lemma \ref{lm52}, we get the integrability of the series $a(\T)\boxast_0 b(\U)$. The theorem is then follows from Lemma \ref{lm54} and \ref{lm55}.
\end{proof}

%--------------------
\subsection{Motivic reflexion formulas}
In this paragraph, we formulate the motivic reflexion formulas for the multivariate case that generalize the motivic Euler reflexion formula. As a consequence, we show that the $\boxast$-product is associative in the class of motivic multiple zeta functions defined in Definition \ref{MMZF}. A corollary of the associativity will be also given.

\begin{theorem}\label{associativity}
Let $\f=(f_1,\dots, f_r)$ and $\g=(g_1,\dots, g_s)$ be ordered families of regular functions on smooth algebraic $k$-varieties $X_1,\dots, X_r$ and $Y_1,\dots, Y_s$, respectively. Then
\begin{equation*}
\zeta_{\f}(\T)\boxast \zeta_{\g}(\U)
=\sum \iota^*\zeta_{p_1,\dots, p_{\eta}}(T_{\alpha_1}^{a_1}U_{\beta_1}^{b_1},\dots,T_{\alpha_{\eta}}^{a_{\eta}}U_{\beta_{\eta}}^{b_{\eta}}),
\end{equation*}
where the sum is taken over all the ordered families of regular functions $(p_1,\ldots,p_{\eta})$ satisfying 
$$p_i=a_i f_{\alpha_i}\oplus b_i g_{\beta_i},\ 1\leq i\leq \eta,$$
with $(a_i,b_i)\in \{0,1\}^2\setminus \{(0,0)\}$, $\sum(a_i+b_i)=r+s$, and $\{\alpha_i\}_{a_i=1}$ and $\{\beta_i\}_{b_i=1}$ being strictly monotonic increasing sequences; $\iota$ is the inclusion of $X_0\times Y_0$ in $X\times Y$ (cf. (\ref{XY})). 
\end{theorem}

\begin{proof}
First, we note that $\zeta_{\f}(\T)$ and $\zeta_{\g}({\U})$ are elements of $\MM_{X_0}^{<}[[\T]]$ and $\MM_{Y_0}^{<}[[{\U}]]$, respectively. By definition, it suffices to show that
\begin{align}\label{sufficient}
\Phi(\zeta_{\f}(\T))\boxast_0 \Phi(\zeta_{\g}(\U))=\sum_{\p}\Phi(\iota^*\widetilde\zeta_{\p}),
\end{align}
where $\p=(p_1,\ldots,p_{\eta})$, $\widetilde\zeta_{\p}=\zeta_{\p}(T_{\alpha_1}^{a_1}U_{\beta_1}^{b_1},\dots,T_{\alpha_{\eta}}^{a_{\eta}}U_{\beta_{\eta}}^{b_{\eta}})$, and the sum is taken over all the $\p$ in the theorem. Writing $\Phi(\zeta_{\f}(\T))=\sum_{\n\in \Delta_{<}} a_{\n}\T^\n$ and $\Phi(\zeta_{\g}(\U))=\sum_{\m\in \Delta_{<}} b_{\m} \U^\m$ we get
\begin{align*}
a^{(i)}_{\n}&=\left[\left\{\varphi\in\mathscr{L}_{n_i}(X_i) \mid f_i(\varphi)=t^{n_i} \mod t^{n_i+1}\right\}\to X_{i,0}\right] \mathbb{L}^{-d_in_i},\\
b^{(j)}_{\m}&=\left[\left\{\psi\in\mathscr{L}_{m_j}(Y_j) \mid g_j(\psi)=t^{m_j} \mod t^{m_j+1}\right\}\to Y_{j,0}\right] \mathbb{L}^{-e_j m_j},
\end{align*}
with $d_i=\dim_k X_i$ and $e_j=\dim_k Y_j$. 

Observe that the coefficients of $\T^\n\U^\m$ in both sides of (\ref{sufficient}) are zero for $\n\not\in \Delta_<$ or $\m\not\in \Delta_<$. In this case, indeed, the statement for the left hand side comes from Definition \ref{n-product} (i), and that for the right hand side is due to the hypothesis that the sequences $\{\alpha_i\}_{a_i=1}$ and $\{\beta_i\}_{b_i=1}$ are strictly monotonic increasing. For $\n\in \Delta_<$ and $\m\in \Delta_<$, since the supports of the $\widetilde\zeta_{\p}$ are distinct, it suffices to show that there exists $\p$ such that the coefficient of $\T^\n\U^\m$ in $\Phi(\iota^*\widetilde\zeta_{\p})$ equals the one in $\Phi(\zeta_{\f}(\T))\boxast_0 \Phi(\zeta_{\g}(\U))$. To prove this, we set 
$$\{l_1<\dots<l_{\eta}\}:=\{\n\}\cup \{\m\}=\{n_1,\dots,n_r, m_1,\dots,m_s\}$$ 
and set 
$$p_i=a_i f_{\alpha_i}\oplus b_i g_{\beta_i},\ 1\leq i\leq \eta,$$ 
with $a_i=1$ (resp. $b_i=1$) if $l_i=n_{\alpha_i}\in \{\n\}$ (resp. $l_i=m_{\beta_i}\in \{\m\}$), otherwise $a_i=0$ (resp. $b_i=0$). Define $\l:=(l_1,\dots,l_{\eta})$. It is easily checked that the coefficient $c_{\l}$ of $\T^\n\U^\m$ in $\Phi(\iota^*\widetilde\zeta_{\p})$  equals $c^{(1)}_{\l}\times\cdots\times c^{(\eta)}_{\l}$, where 
\begin{align*}
c^{(i)}_{\l}:=\left[\left\{\omega\in\mathscr{L}_{l_i}(Z_i) \mid p_i(\omega)=t^{l_i} \mod t^{l_i+1}\right\}\to Z_{i,0}\right] \mathbb{L}^{-\delta_i l_i}
\end{align*}
with $Z_i:=(X_{\alpha_i})^{a_i}\times (Y_{\beta_i})^{b_i}$ and $\delta_i=\dim_k Z_i$. It follows from the proof of Theorem \ref{thm21} and direct calculations that 
$$c^{(i)}_{\l}=
\begin{cases} a^{(\alpha_i)}_{\n} & \text{ if } b_i=0,\\
b^{(\beta_i)}_{\m} & \text{ if }a_i=0,\\
a^{(\alpha_i)}_{\n}\ast b^{(\beta_i)}_{\m} & \text{ if }a_i=b_i=1.
\end{cases}
$$
This proves the theorem.
\end{proof}

The following corollaries are direct consequences of Theorem \ref{associativity}.
\begin{corollary}\label{lastcor1}
Let $\f=(f_1,\dots,f_r)$, $\g=(g_1,\dots,g_s)$ and $\mathbf{h}=(h_1,\dots,h_{\tau})$ be ordered families of regular functions on smooth algebraic $k$-varieties $X_1,\dots,X_r$, $Y_1,\dots,Y_s$ and $Z_1,\dots,Z_{\tau}$, respectively. Then
\begin{align*}
(\zeta_{\f}(\T)\boxast \zeta_{\g}(\U))\boxast \zeta_{\mathbf h}(\mathbf V)=\sum\iota^*\zeta_{p_1,\ldots,p_{\tau}}(T_{\alpha_1}^{a_1}S_{\beta_1}^{b_1}U_{\gamma_1}^{c_1},\ldots,T_{\alpha_\tau}^{a_\tau}S_{\beta_\tau}^{b_\tau}U_{\gamma_\tau}^{c_\tau}),
\end{align*}
where the sum is taken over all the ordered families of regular functions $(p_1,\dots,p_{\eta})$ satisfying 
$$p_i=a_i f_{\alpha_i}\oplus b_i g_{\beta_i}\oplus c_i h_{\gamma_i}, \ 1\leq i\leq \eta,$$ 
with $(a_i,b_i,c_i)\in \{0,1\}^3\setminus \{(0,0,0)\}$, $\sum(a_i+b_i+c_i)=r+s+\tau$, and $\{\alpha_i\}_{a_i=1}$, $\{\beta_i\}_{b_i=1}$ and $\{\gamma_i\}_{c_i=1}$ being strictly monotonic increasing sequences. 

In particular, the $\boxast$-product is associative in the class of motivic multiple zeta functions.
\end{corollary}

%\begin{corollary}%[Commutativity and associativity]
%The $\boxast$-product is commutative and associative in the class of motivic zeta functions defined in Definition \ref{MMZF}.
%\end{corollary}

\begin{corollary}\label{lastcor2}
Let $f$, $g$ and $h$ be regular functions on smooth algebraic $k$-varieties $X$, $Y$ and $Z$, respectively. Then, up to the pullback of an inclusion of $X_0\times Y_0\times Z_0$ in a Zariski closed subset of $X\times Y\times Z$, the following identity holds in $\mathscr M_{X_0\times Y_0\times Z_0}^{\hat\mu}[[T,U,V]]$:
\begin{align*}
\zeta_{f}(T)\boxast \zeta_{g}(U)\boxast \zeta_{h}(V)&=\zeta_{f,g,h}(T,U,V)+\zeta_{f,h,g}(T,V,U)+\zeta_{g,h,f}(U,V,T)\\
&\quad\quad +\zeta_{g,f,h}(U,T,V)+\zeta_{h,f,g}(V,T,U)+\zeta_{h,g,f}(V,U,T)\\
&\quad\quad +\zeta_{f\oplus g,h}(TU,V)+\zeta_{h,f\oplus g}(V,TU)+\zeta_{f,g\oplus h}(T,UV)\\
&\quad\quad +\zeta_{g\oplus h,f}(UV,T)+\zeta_{g,f\oplus h}(U,TV)+\zeta_{f\oplus h,g}(TV,U)\\
&\quad\quad +\zeta_{f\oplus g\oplus h}(TUV).
\end{align*}
\end{corollary}

\begin{remark}
After many attempts we still do not know whether the $\boxast$-product is associative in the class of integrable series over monodromic Grothendieck rings of algebraic varieties.
\end{remark}

% References   ---------------------------------------------------------------------------------
\bibliographystyle{amsplain}

\end{document}